\documentclass[reqno]{amsart}
\usepackage{a4wide}
\usepackage{graphics,graphicx}
\usepackage[english]{babel}
\usepackage{color}
\usepackage{amsfonts,amsthm,amssymb,amsmath,mathrsfs}
\usepackage[colorlinks=false, linktocpage=true]{hyperref}
\newtheorem{thm}{Theorem}[section]

\newtheorem{lem}[thm]{Lemma}
\newtheorem{prop}[thm]{Proposition}
\theoremstyle{definition}

\theoremstyle{remark}
\newtheorem{rem}[thm]{Remark}
\numberwithin{equation}{section}

\newcommand{\Real}{\mathbb{R}}

\newcommand{\med}{\textnormal{med}}

\begin{document}
	\title[Control and Stabilization for the DGB equation]{Control and Stabilization for the dispersion generalized  Benjamin equation on the Circle}

	\author{Francisco J. Vielma Leal}\address{IMECC-UNICAMP, Rua Sérgio Buarque de Holanda, 651
		13083-859, Campinas, SP, Brasil}
	\email{fvielmaleal7@gmail.com; vielma@ime.unicamp.br}
	
		\author{Ademir Pastor }\address{IMECC-UNICAMP, Rua Sérgio Buarque de Holanda, 651, 13083-859, Campinas, SP, Brasil}
		\email{ apastor@ime.unicamp.br}


\begin{abstract} 
	This paper is concerned with controllability and stabilization properties of the dispersion generalized Benjamin equation on the periodic domain $\mathbb{T}.$ First, by assuming the control input  acts on all the domain, the system is proved to be globally exactly controllable in the Sobolev space $H^{s}_{p}(\mathbb{T}),$ with $s\geq 0.$ Second, by providing a locally-damped term added to the equation as a feedback law, it is shown that the resulting equation is globally well-posed and locally exponentially stabilizable in the  space $L^{2}_{p}(\mathbb{T}).$  The main ingredient to prove the global well-posedness is the introduction of  the dissipation-normalized Bourgain  spaces which allows one to gain smoothing properties simultaneously from the dissipation and  dispersion present in the equation. Finally,  the local exponential stabilizability result is accomplished taking into account the decay of the associated semigroup combined with the fixed point argument.   
\end{abstract}

\maketitle

\section{Introduction}\label{intr}
We consider the following dispersion generalized Benjamin (DGB) equation 
\begin{equation}\label{GB}
	\partial_{t}u+\beta D^{2m}\partial_{x}u+\alpha \mathcal{H}^{2r}\partial_{x}u+\partial_{x}(u^2)=0,
\end{equation}
where $u=u(x,t)$ denotes a real-valued function of the real variables $x$ and $t$, the constants  $\alpha$ and $\beta$ are non-negative, and the Fourier multiplier operators $D^{2m}$ and $\mathcal{H}^{2r}$ are defined by $\widehat{D^{2m}u}(\xi):=|\xi|^{2m}\widehat{u}(\xi)$ and $\widehat{\mathcal{H}^{2r}u}(\xi):=-|\xi|^{2r}\widehat{u}(\xi),$ respectively. Here $m$ and $r$ are real constants with $0<r<m.$ 

Equation \eqref{GB} encompass several well known dispersive models. In fact, if we set $m=1$ and $\alpha=0$ we obtain the Korteweg-de Vries equation (KdV), which was introduced in \cite{kdv}  as an approximate model for planar, unidirectional, irrotational waves propagating on the surface of shallow water. If we set  $r=\frac{1}{2}$ and $\beta =0$ then the equation reduces to the Benjamin-Ono equation (BO) which was deduced in \cite{Benjamin0} and \cite{ono} to model  one-dimensional internal waves in deep water. Also, in the case $m=1$ and $r=\frac{1}{2}$ we obtain  the Benjamin equation, which was derived in \cite{Benjamin} to describe the evolution of waves on the interface of a two-layer system of fluids  in which surface tension effects are not negligible. Thus, equation \eqref{GB} can be thought as an hybrid model between the KdV  and BO equations in the sense that the second and third terms  are dispersion generalizations of the KdV and BO, respectively.

Both KdV and BO equations are one of the most studied dispersive models; several results concerning well-posedness and the behavior of solutions may be found in the current literature. Since the references are too extensive, we refrain from list them at this point and focus our attention into the DGB equation. The generalized version of \eqref{GB},  specifically when the nonlinearity is replaced by $\partial_{x}(u^{p+1})$ for some integer number $p\geq1,$  has been object of study in recent years. 
In \cite{Linares and Scialom}, the authors considered the initial-valued problem (IVP) associated to that equation to understanding the interaction between the dispersive term and the nonlinearity in the context of the theory of nonlinear dispersive evolution equations. In particular, they proved that this equation with $m>1$ and $\beta \neq 0$ is globally well-posed in its energy space which is the natural space to study properties of the solitary-wave solutions. Comparison between the solutions of \eqref{GB} and that of the generalized BO equation ($\beta=0$) as well as the behavior of the solutions in the limiting case ($\alpha\to0$) were also addressed.  In \cite{Bona Chen} the authors indeed proved the existence of solitary-wave solutions; they showed that solitary waves do exist for any $p\geq1$ and $\alpha$ small (with $\beta\neq0$). Asymptotic and decay properties of the solitary waves were also established. For the generalized Benjamin equation ($m=1$ and $r=\frac{1}{2}$), existence and stability of solitary waves whether $p=1$ were established in \cite{Benjamin} and \cite{Benjamin1} in the case of higher velocities. Later, using the known results for the KdV equation, in \cite{Albert} it was proved the existence and stability of solitary waves for $\alpha$ sufficiently small. In addition, the stability (for $1<p<4$) and instability (for $p>5$)  of solitary waves were addressed in \cite{Albert1} and \cite{Angulo}, respectively. We also recall that global well-posedness for the Benjamin equation in $L^2(\mathbb{R})$ was established in \cite{Linares}. These global result was sharpened to $H^s(\mathbb{R})$, $s\geq-3/4$ in \cite{Chen}. In the periodic context, global well-posedness in $L^2(\mathbb{T})$ was also established in \cite{Linares}. More recently, in \cite{Shi} the authors proved the (sharp) local well-posedness in the Sobolev spaces $H^s(\mathbb{T})$, $\geq-1/2$, for small data. Also, existence and stability of periodic traveling-wave solutions were obtained in \cite{Angulo1}.

\subsection{Problems under consideration.}\label{intr1}	

In this paper, we are interested in studying the controllability and stabilization problems associated to the DGB equation posed on the periodic domain $\mathbb{T}:=\mathbb{R}\text{/}2\pi \mathbb{Z}$. Hence, we consider the IVP
\begin{equation}\label{GB2}
\begin{cases}
		\partial_{t}u+\beta D^{2m}\partial_{x}u+\alpha \mathcal{H}^{2r}\partial_{x}u+\partial_{x}(u^2)=0, \quad x\in \mathbb{T},\; t>0,\\
		u(x,0)=u_{0}(x), 
	\end{cases}
\end{equation}
where now the Fourier multipliers $D^{2m}$ and $\mathcal{H}^{2r}$ are defined as $$\widehat{D^{2m}u}(k)=|k|^{2m}\widehat{u}(k), \quad \widehat{\mathcal{H}^{2r}u}(k)=-|k|^{2r}\widehat{u}(k), \quad k\in\mathbb{Z}.
$$
Specifically, we will address the  following two problems:\\

\textbf{\emph{Exact controllability problem:}} Let $T>0$ be given. Given an initial state  $u_{0}$ and a terminal state $u_{1}$ in a certain space.
Can one find an appropriate control input $f$ such that the IVP  
\begin{equation}\label{cntl}
\begin{cases}
\partial_{t}u+\beta D^{2m}\partial_{x}u+\alpha \mathcal{H}^{2r}\partial_{x}u+\partial_{x}(u^2)=f(x,t), & x\in \mathbb{T},\; t\in (0,T),\\
u(x,0)=u_{0}(x),
\end{cases}
\end{equation}
admits a solution $u$ which satisfies
$u(T)=u_{1}$?\\

\textbf{\emph{Asymptotic stabilizability problem:}} Let  $u_{0}\in L_{p}^{2}(\mathbb{T})$ be given.
Can one define a linear feedback control law $B$ such that the resulting closed-loop system
$$	\begin{cases}
		\partial_{t}u+\beta D^{2m}\partial_{x}u+\alpha \mathcal{H}^{2r}\partial_{x}u+\partial_{x}(u^2)=Bu, & x\in \mathbb{T},\;\; t\in \mathbb{R}^{+},\\
		u(x,0)=u_{0}(x),
	\end{cases}$$
is globally well-defined and asymptotically stable  to an equilibrium point as $t\rightarrow +\infty$?\\

Dispersive equations in the periodic context have attracted attention of both mathematicians and physicists in recent years. In particular, controllability and stabilization  on a periodic domain have been widely studied taking into account the smoothing properties of the linear part of the equation in the well known Bourgain spaces. Indeed, let us recall some works closely related to our purposes. For small data, the controllability and stabilization for the KdV  equation was first studied in \cite{10}; such results were extended to arbitrarily large data in \cite{14}. Since the strength of the dispersion in the BO equation is too weak when compared to the nonlinearity, the problems of controllability and stabilization for BO is even more delicate and it was studied only recently in \cite{Linares Rosier}, where the authors established a semi-global
stabilization result  of weak solutions in $L^2_p(\mathbb{T})$ and the local exponential stability in
$H^s_p(\mathbb{T})$, $s>1/2$. After that, in \cite{Laurent1} the authors extended these results to global ones in $L^2_p(\mathbb{T})$. For the Benjamin equation, controllability and stabilization were addressed in \cite{panthee} and \cite{panthee1}.

The  dispersion generalized Benjamin-Ono equation 
\begin{equation}\label{GBO}
	\partial_{t}u+ D^{2m}\partial_{x}u+\partial_{x}(u^2)=0,
\end{equation}
also appears as an important dispersive model and it has been widely studied. In a physical context, it  models vorticity waves in coastal zones \cite{Shrira}. In the case $m$ is a positive integer (so that \eqref{GBO} can also be viewed as a high-order KdV equation), local and global control and stabilization properties were established in \cite{zhao} and  \cite{Capistrano Kawak Vielma}. In the case $m$ is a real number satisfying  $m \in (1/2,1),$ stabilization was investigated  in \cite{Flores OH and Smith}, where the authors showed that \eqref{GBO} is globally well-posed and locally exponentially stabilizable in $L^{2}_{p}(\mathbb{T}).$  The feedback control law was constructed as follows: take a real non-negative smooth function $g=g(x)$  defined on $\mathbb{T}$ such that $\omega:=\text{supp}(g)=\{x\in \mathbb{T}: g(x)>0\}$ is an open interval and 
\begin{equation}\label{OPG1}
	2\pi [g]=\int_{\mathbb{T}}g(x)dx=1,
\end{equation}
where $[g]$ represents the mean value of the function $g.$ Define the operator $G$ as
\begin{equation}\label{Gdef}
G(\phi)(x):=g(x)\left(\phi(x)- \int_{\mathbb{T}} \phi(y) g(y) dy\right),\;\;\phi\in L^{2}_{p}(\mathbb{T}),
\end{equation}
Then, to keep the mass conserved they choose the feedback control law to be of the form 
$$	Bu=-GD^{\delta}Gu,$$
for some positive real $\delta$.

Our strategy here to establish the stabilization results will follow closely the one in \cite{Flores OH and Smith}. Throughout the paper we assume $m>\frac{1}{2}$ and $0<r<m$. Although the most interesting and difficult case is $m\in(\frac{1}{2},1)$ where the dispersion is below that of the KdV equation, we assume $m>\frac{1}{2}$ in order to keep in mind that our results may be seen as an extension of the ones in \cite{Flores OH and Smith} (and also in \cite{zhao}) to the model \eqref{GB}.   Therefore, one of our main interest is to study the stabilization properties of a locally-damped variant of \eqref{GB2} given by
\begin{equation}\label{GB3}
	\begin{cases}
		\partial_{t}u+\beta D^{2m}\partial_{x}u+\alpha \mathcal{H}^{2r}\partial_{x}u+\partial_{x}(u^2)=-GD^{\delta}Gu, \quad x\in \mathbb{T},\; t>0,\\
		u(x,0)=u_{0}(x), 
	\end{cases}
\end{equation}
on $L^{2}_{p}(\mathbb{T}).$ Note that in view of the signs of the operators $D^{2m}$ and $\mathcal{H}^{2r}$ both terms in the linear part of \eqref{GB3} are competing each other; so that we have an extra difficulty to control. In addition, as we will see below, the eigenvalues associated with the linear part of the equation may have multiplicities; so, our analysis needs extra efforts to deal with multiple eigenvalues.

\subsection{Main results.}\label{intr2}

Let us start with the controllability result. As for the BO equation, the controllability problem for \eqref{GB2} in $H^{s}_{p}(\mathbb{T})$ turns out to be a very difficult and challenging problem.  However, if we do not require to have an exact controllability result in its maximum strength as was done, for instance,  in \cite{14} for the KdV equation, we may consider the control function $f$ in \eqref{cntl} to be of the form $f(x,t)= G(h)(x,t)$  with an additional condition on the function $g$ defined in \eqref{OPG1} (see condition \eqref{NLcontro2} and Remark \ref{NLcontrol1}) in the same spirit of  \cite{10,1}. More precisely, in order to get an exact control result, $h$ may be considered as the new control function and we suppose that $\omega=\textrm{supp}(g)=\mathbb{T}.$ Hence, the control function $f\equiv G(h)$ acts on all $\mathbb{T}.$ In others words, we have global exact controllability in the global control case as stated in the following result.
\begin{thm}\label{NLcontrol}
	Let $\alpha>0,$ $\beta>0,$ and $s\in \mathbb{R}$ with $s\geq 0$ be given. Assume that
	\begin{equation}\label{NLcontro2}
		g(x)=\frac{1}{2\pi},\;\;\text{for all}\;x\in \mathbb{T}.
	\end{equation}
	Let $T>0$ be given, then for any $u_{0},u_{1}\in H^{s+1}_{p}(\mathbb{T}),$ with $[u_{0}]=[u_{1}]$ there exists a function $h\in L^{2}(0,T;H_{p}^{s}(\mathbb{T}))$ such that the solution $u\in C([0,T];H_{p}^{s+1}(\mathbb{T}))$ of the nonlinear system 
		$$\begin{cases}
			\partial_{t}u+\beta D^{2m}\partial_{x}u+\alpha \mathcal{H}^{2r}\partial_{x}u+\partial_{x}(u^2)=Gh,& x\in \mathbb{T},\; t\in[0,T],\\
			u(x,0)=u_{0}(x), &
		\end{cases}$$
	satisfies $u(x,T)=u_{1}(x),\;\;x\in \mathbb{T}.$
\end{thm}

In order to prove Theorem \ref{NLcontrol} we follow the strategy in \cite{1} and \cite{10} which relies on proving that the controllability of the nonlinear problem follows almost immediately from the controllability of the linear problem.
Note that we have taken the data $u_0$ and $u_1$ in $H_p^{s+1}(\mathbb{T})$ instead of $H_p^{s}(\mathbb{T})$; this implies that once we have solved the linear problem in $H_p^{s+1}(\mathbb{T})$ we promptly obtain that $\partial_{x}(u^2)$ belongs to $H_p^{s}(\mathbb{T})$. We point out that it would be interesting to obtain the exact controllability for data in $H_p^{s}(\mathbb{T})$.

\begin{rem}\label{NLcontrol1}
Eventually, the function $g$ in \eqref{NLcontro2} may be replaced by a more general one satisfying \eqref{OPG1} and
	$$|g(x)|>g_0>0,\;\;\text{for all}\;x\in \mathbb{T},$$
	for some constant $g_0\in \mathbb{R}$ (see \cite{1} and \cite{10}).
\end{rem}

Next we pay particular attention to the stabilization problem in the space  $L_{p}^{2}(\mathbb{T})$. First of all we establish the global well-posedness of the IVP \eqref{GB3}.

\begin{thm}\label{GWP}
	Let $\alpha>0,$ $\beta>0,$ $r>0,$ $m>\frac{1}{2},$ with $r<m$  and $T>0$ be given.	 Then for any given $u_{0}\in L^{2}_{p}(\mathbb{T})$ and $0<\delta\leq 1$ with $\max\{0, 2-2m\}<\delta$ (and therefore $\delta<2m$)  the IVP \eqref{GB3} admits a unique solution $u\in C([0,T];L_p^{2}(\mathbb{T})).$ Moreover, the solution map $u_{0}\in L^{2}_{p}(\mathbb{T})\longmapsto u(t)\in C([0,T];L_p^{2}(\mathbb{T}))$ is uniformly continuous within a bounded set of $L^{2}_{p}(\mathbb{T}).$
\end{thm}
Finally, we obtain the following exponential stabilization result, which gives and affirmative answer to the stabilization problem.

\begin{thm}\label{Sta1}
	Under the assumptions of Theorem \ref{GWP}, there exist $\rho>0$ and $\lambda>0$ such that for any $u_{0}\in L^{2}_{p}(\mathbb{T})$ with $\|u_{0}\|_{L_{p}^{2}(\mathbb{T})}<\rho$, the unique solution $u\in C([0,+\infty);L_p^{2}(\mathbb{T}))$ of system \eqref{GB3} satisfies
	$$\|u(\cdot, t)-[u_{0}]\|_{L^{2}_{p}(\mathbb{T})}\leq M e^{-\lambda t}\|u_{0}-[u_{0}]\|_{L^{2}_{p}(\mathbb{T})},$$
for all $t\geq0$ and some positive constant $M$.
\end{thm}

In order to prove Theorems \ref{GWP} and \ref{Sta1} we will assume that the initial data has zero mean, which, essentially,  adds a linear term in the equation. Indeed, note that for any solution $u$ of  \eqref{GB3} its mean value $[u]$ is invariant, that is,
$$ [u]=\frac{1}{2\pi}\int_{0}^{2\pi}u(x,t)dx=\frac{1}{2\pi}\int_{0}^{2\pi}u_{0}(x)dx = [u_0].$$
 Thus, by introducing the number $\mu:=[u(\cdot,t)]=[u_{0}]$ and making the change of variable
$$v=u-\mu.$$
we see that  $[v]=0$ and solving \eqref{GB3} is equivalent to solve	
\begin{equation}\label{GB6}
	\begin{cases}
		\partial_{t}v+\beta D^{2m}\partial_{x}v+\alpha \mathcal{H}^{2r}\partial_{x}v+2\mu  \partial_{x}v+  \partial_{x}(v^2)=-GD^{\delta}Gv,& x\in \mathbb{T},\; t>0,\\
		v(0)=v_{0}, &
	\end{cases}
\end{equation}
From now on, $\mu$ will denote a given real constant and we shall establish the well-posedness and exponential stability results in $L^{2}_{0}(\mathbb{T})$ for the problem  \eqref{GB6}. The main ingredient to prove the results is the introduction of the dissipation-normalized Bourgain spaces; such spaces allow one to get smoothing via dispersion and dissipation.

\begin{rem}
	Throughout the paper we will assume $\alpha>0$ and $\beta>0$. However, Theorems \ref{NLcontrol}, \ref{GWP} and \ref{Sta1} also hold if we assume either $\alpha=0$ or $\beta=0$; in this case many of the proofs may be simplified since the linear terms are not competing each other.
\end{rem}

\subsection{Structure of the paper.}\label{intr3}
Some preliminaries are given in section \ref{preliminares}, where we set some notations and introduce the Sobolev spaces $H_{p}^{s}(\mathbb{T})$ of order $s\in \mathbb{R}.$  In Section \ref{Linear Systems}, we study the control and stabilization properties for the respective linearized system. In Section \ref{Proof}, we provide the proof of Theorem \ref{NLcontrol} which is a direct consequence of the exact controllability result presented for the linearized system. In Section \ref{Bourgspace}, we introduce the dissipation-normalized Bourgain space associated to the DGB equation and  establish some linear and integral estimates; we finalize that section with the proof of a bilinear estimate, which is the main ingredient to prove our well-posedness result. The  existence of global solutions for the IVP \eqref{GB3} is presented in Section \ref{GWP1}. Finally, in Section \ref{LES}, we show the local exponential stabilization result stated in Theorem \ref{Sta1}.

\section{Preliminaries}\label{preliminares}
In this section we introduce some basic notations and recall the main tools to obtain our results. Given two positive constants $a$ and $b$ we use $a\lesssim b$ to indicate that $a\leq Cb$ for some positive constant $C$; also, we use $a \lesssim_{X,\ldots,Y} b$ to say that the implicit constant $C$ depends on the parameters $X,\ldots,Y$. Also, $a\sim b$ means that $C^{-1}a\leq b\leq Ca$. We denote by $\mathscr{P}$ the space $C^\infty_p(\mathbb{T})$  of all $C^\infty$  functions that are $2\pi$-periodic. By $\mathscr{P}'$ (the topological dual of $\mathscr{P}$) we denote the space of all periodic distributions.  By $L^2_p(\mathbb{T})$ we denote the standard space of the square integrable $2\pi$-periodic functions.

 The Fourier transform of $v\in \mathscr{P}'$ is the sequence $\{\widehat{v}(k)\}_{k\in\mathbb{Z}}$ defined as
$$	\widehat{v}(k)=\frac{1}{2 \pi}\langle v,e^{-ikx}\rangle,\;\; \;k \in \mathbb{Z}.$$
Let $\mathscr{S}'(\mathbb{Z})$ denote the space of the sequences with slow growth. The map $^\wedge:\mathscr{P}'\to\mathscr{S}'(\mathbb{Z})$ is a linear bijection with inverse \;$^{\vee}:\mathscr{S}'(\mathbb{Z})\rightarrow \mathscr{P}'$   (the inverse Fourier transform) defined by
$$\alpha=\{\alpha_{k}\}_{k\in \mathbb{Z}}\mapsto
\alpha^{\vee}(x):=\sum_{k\in\mathbb{Z}}\alpha_{k}e^{ikx},$$
and the series converges  in the sense of $\mathscr{P}'$.

Next we introduce the periodic Sobolev spaces. For a more detailed description and properties of
these spaces, we refer the reader to \cite{6}. Given $s\in\Real$, the (periodic) Sobolev space of order $s$ is defined as
$$H^{s}_{p}(\mathbb{T})=\left\{ v\in \mathscr{P}' :\; \|v\|_{H^{s}_{p}(\mathbb{T})}^{2}:=2\pi\sum_{k\in\mathbb{Z}}(1+|k|)^{2s}|\widehat{v}(k)|^{2}<\infty \right\}.$$
The space $H^{s}_{p}(\mathbb{T}) $ is a Hilbert space endowed with the inner product
\begin{equation*}
	\displaystyle{(u, v)_{H^{s}_{p}(\mathbb{T})}= 2\pi  \sum_{k\in \mathbb{Z}}
		(1+|k|)^{2s}\widehat{u}(k)\;\overline{\widehat{v}(k)}.}
\end{equation*}
For any $s\in \Real$, $(H^{s}_{p}(\mathbb{T}))'$, the topological dual of $H^{s}_{p}(\mathbb{T})$, is isometrically isomorphic to $H^{-s}_{p}(\mathbb{T})$, where the duality is implemented by the pairing
$$
\displaystyle{\langle h, v\rangle_{H^{-s}_{p}(\mathbb{T})\times H_{p}^{s}(\mathbb{T})}= 2\pi \sum_{k\in \mathbb{Z}}
	\widehat{h}(k)\;\overline{\widehat{v}(k)}},\;\;\text{for all}\; v \in H^{s}_{p}(\mathbb{T}),\;h\in H^{-s}_{p}(\mathbb{T}).
$$

	It may be proved that any periodic distribution $v\in \mathscr{P}'$ may be written as (see, for instance, \cite[page 188]{6})
	\begin{equation}\label{prep}
		v=\sqrt{2\pi}\sum_{k\in\mathbb{Z}}\widehat{v}(k)\psi_{k},
	\end{equation}
	where
	\begin{equation}\label{spi}
		\psi_{k}(x):=\frac{e^{ikx}}{\sqrt{2 \pi}},\;\; k\in \mathbb{Z},
	\end{equation}
	and the series converges in the sense of $\mathscr{P}'$. In particular, any $v\in H^{s}_{p}(\mathbb{T})$, $s\in\Real$, can be written in the form \eqref{prep}.

We also consider the closed subspace
$$H^{s}_{0}(\mathbb{T}):=\left\{ v\in H^{s}_{p}(\mathbb{T)}\left|\right. \;\; \widehat{v}(0)=0\right\}.$$
If $s_{1}, s_{2}\in \mathbb{R}$ with $s_{1}\geq s_{2}$ then
$H_{0}^{s_{1}}(\mathbb{T})$ is densely embedded in $H_{0}^{s_{2}}(\mathbb{T})$. Since $H^0_p(\mathbb{T})$ is isometrically isomorphic to $L^2_p(\mathbb{T})$ we shall denote $H_{0}^{0}(\mathbb{T})$ by $L_{0}^{2}(\mathbb{T}).$
Note that $L_{0}^{2}(\mathbb{T})$ is a closed subspace of
$L^{2}_{p}(\mathbb{T})$.

\section{Linear Systems}\label{Linear Systems}

This section is devoted to study the controllability and stabilization problems for the linear DGB equation. Thus, consideration is given to the associate linear open-loop control system
\begin{equation}\label{GB4}
	\begin{cases}
		\partial_{t}v+\beta D^{2m}\partial_{x}v+\alpha \mathcal{H}^{2r}\partial_{x}v+2\mu \partial_{x}v
		=Gh,\quad x\in \mathbb{T},\; t\in [0,T],\\
		v(x,0)=v_{0}(x), 
	\end{cases}
\end{equation}
where $v=v(x,t)$ denotes a real-valued function, $\mu$ is a real constant,  $h=h(x,t)$ is the applied control function and
the operator $G$ is the bounded linear  operator defined by
\begin{equation}\label{EQ1}
	G(\phi):=g\phi-g\,\langle\phi,g\rangle, \qquad \phi\in H_{p}^{s}(\mathbb{T}),
\end{equation}
where the first product must be understood in the periodic distributional sense and $\langle\cdot,\cdot\rangle$ denotes the pairing between $\mathscr{P}'$ and $\mathscr{P}$ (see Remark 1.2 in \cite{Pastor and Vielma}). Note that if $s\geq0$ then $G$ is exactly the operator in \ref{Gdef}. Recall we are always assuming $\alpha>0$, $\beta>0$, $m>\frac{1}{2}$ and $0<r<m$.

Let $\mathcal{A}:\mathcal{D}(\mathcal{A})\subset L_p^{2}(\mathbb{T}) \to L_p^{2}(\mathbb{T})$ denote the multiplier operator $\mathcal{A}\varphi=-\beta D^{2m}\varphi- \alpha H^{2r}\varphi-2\mu \varphi$ with domain $\mathcal{D}(\mathcal{A})=H_p^{2m}(\mathbb{T}).$ Then $\mathcal{A}$ has  order $2m$ and symbol $a:\mathbb{Z}\to \mathbb{R}$ given by
\begin{equation*}
	a(k):=-\beta |k|^{2m}+ \alpha |k|^{2r}-2\mu.
\end{equation*}
Because  $0<r<m,$ it is easy to see that
\begin{equation*}
	|a(k)|\leq C |k|^{2m}, \;\;\text{for all}\;k\in \mathbb{Z}-\{0\},
\end{equation*}
for some positive constant $C.$

In what follows we will show that we can apply the results in  \cite[Theorem 1.3 and Remark 1.4]{Pastor and Vielma} in order to prove that \eqref{GB4} is exactly controllable for any positive time $T>0$ and exponentially stabilizable  with any given decay rate in the Sobolev spaces. Indeed, first of all  we note that the operator $\partial_{x}\mathcal{A}$ is skew-adjoint in $H_{p}^{s}(\mathbb{T}),$ for any $s\in \mathbb{R}$, that is,
\begin{equation}\label{skew}
	(\partial_{x}\mathcal{A})^{\ast}
	=
	-\partial_{x}\mathcal{A}.
\end{equation}

Using the Fourier transform, it is easy to check that 
the following property holds:
\vskip.2cm
\begin{itemize}
	\item [$(H1)$] $\partial_{x}\mathcal{A}\psi_{k}=i\lambda_{k} \psi_{k},$ where  $\psi_{k}$ is defined in \eqref{spi} and  $\lambda_{k}:=k a(k)=-\beta k|k|^{2m}+ \alpha k |k|^{2r}-2\mu k,$ for all $k\in \mathbb{Z}.$
\end{itemize}
Note that  the eigenvalues in the sequence $\{i \lambda_{k}\}_{k\in \mathbb{Z}}$ are not necessarily distinct. Since we need to distinguish simple and multiples eigenvalues, for each $k_{1}\in \mathbb{Z},$ we set $I(k_{1}):=\{k\in \mathbb{Z}: \lambda_{k}=\lambda_{k_{1}}\}$ and $m(k_{1}):=\#I(k_{1}),$
where $\#I(k_{1})$ denotes the number of elements in $I(k_{1}).$ In particular,  $m(k_1)=1$ if $\lambda_{k_1}$ is a simple eigenvalue.
Concerning the quantity  $m(k_{1})$, we can easily verify that
\begin{itemize}
	\item [$(H2)$] $m(k_{1})\leq5,$ for all $k_{1}\in \mathbb{Z}$. 
\end{itemize}
This is a consequence of the fact that $m(k_{1})$ is less than or equal to the number
of integer roots of the equation $-\beta x |x|^{2m}+\alpha x |x|^{2r}-2\mu x= c,$ where $c$ is an arbitrary real number. Furthermore, depending on the different values of parameters $\beta, \alpha,$ and $\mu$ we have that $m(k_{1})$ is less than or equal to $1,\;3$ or $5.$ See Figure \ref{Fig1}. 
\begin{figure}[h!]
\begin{center}
\includegraphics[scale=0.3]{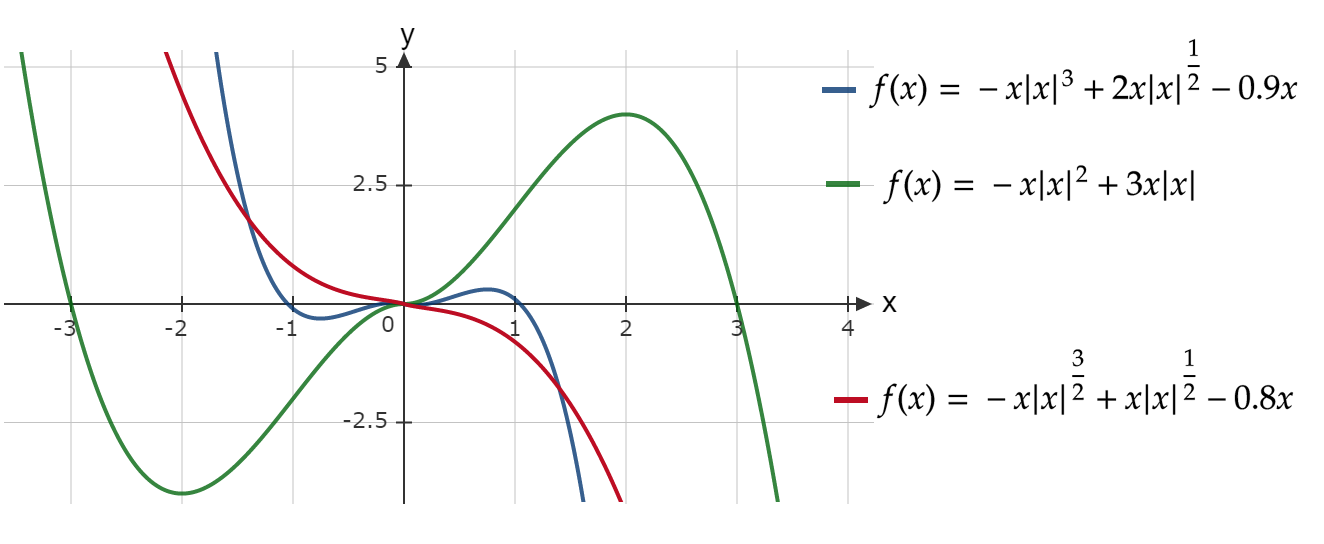}
\caption{Graphs of the function $f(x)=-\beta x |x|^{2m}+\alpha x |x|^{2r}-2\mu x$ for some values of the parameters.}\label{Fig1}
\end{center}
\end{figure}

Next we claim that, for $k$ sufficiently large, 
\begin{equation}\label{incre1}
	\lambda_{k}-\lambda_{k+1}> \alpha (m-r) k^{2r}.
\end{equation}
In fact, 
\begin{equation}\label{incre2}
		\begin{split}
			\lambda_{k}-\lambda_{k+1}&=\beta (k+1)^{2m+1}-\beta k^{2m+1}-\alpha (k+1)^{2r+1}+\alpha k^{2r+1}+2\mu\\
			&=\alpha k^{2r+1}\left[\frac{\beta }{\alpha} k^{2(m-r)} \left(\left(1+\frac{1}{k}\right)^{2m+1}-1 \right)- \left( \left(1+\frac{1}{k}\right)^{2r+1}-1\right)
			+\frac{2\mu}{\alpha k^{2r+1}}\right]\\
			&\geq \alpha k^{2r+1}\left[ \left(\left(1+\frac{1}{k}\right)^{2m+1}-1 \right)- \left( \left(1+\frac{1}{k}\right)^{2r+1}-1\right)
			+\frac{2\mu}{\alpha k^{2r+1}}\right],
	\end{split}
\end{equation}
where we used that, for $k$ sufficiently large, $\frac{\beta }{\alpha} k^{2(m-r)}\geq 1$. Now, from the mean value theorem, for some $\theta\in(2r+1,2m+1)$,
\begin{equation}\label{incre4}
	\begin{split}
		\left(\left(1+\frac{1}{k}\right)^{2m+1}-1 \right)- \left( \left(1+\frac{1}{k}\right)^{2r+1}-1\right)&= \left(1+\frac{1}{k}\right)^{2m+1}-  \left(1+\frac{1}{k}\right)^{2r+1}\\
	&=2(m-r)\left(1+\frac{1}{k}\right)^\theta\ln\left(1+\frac{1}{k}\right)\\
	&\geq \frac{2(m-r)}{k},
\end{split}
\end{equation}
where we used that $\theta>1$ and the fact that $(1+x)\ln(1+x)\geq x$ for any $x>0$.
From \eqref{incre2}-\eqref{incre4}, we obtain 
\begin{equation*}\label{incre3}
		\begin{split}
			\lambda_{k}-\lambda_{k+1}
			\geq \alpha k^{2r+1}\left[ \frac{2(m-r)}{k}
			+\frac{2\mu}{\alpha k^{2r+1}}\right]
			=k^{2r}\left(2(m-r)\alpha
			+\frac{2\mu}{k^{2r}}\right)
	\end{split}
\end{equation*}
which gives \eqref{incre1} for $k$ sufficiently large.

Using \eqref{incre1} and the fact that $\lambda_{-k}=-\lambda_{k},$ for all $k\in \mathbb{Z},$ we conclude that 
\vskip.2cm
\begin{itemize}
	\item[$(H3)$] there exists $k_{1}^{\ast} \in \mathbb{N}$ such that $m(k_{1})=1,$ for all $k_{1}\in \mathbb{Z}$ with $|k_{1}|\geq k_{1}^{\ast}.$ 
\end{itemize}

Therefore, we may count only the distinct eigenvalues to obtain a (maximal) set $\mathbb{I}\subseteq\mathbb{Z}$ and a sequence
$\{\lambda_{k}\}_{k \in \mathbb{I}}$, with the
property that $\lambda_{k_{1}}\neq \lambda_{k_{2}}$, for any $k_{1},k_{2}\in \mathbb{I}$ with $k_{1}\neq k_{2}$. Furthermore, $\{\lambda_{k}\}_{k \in \mathbb{I}}$ is a strictly decreasing sequence for $|k|$ large enough. Additionally, \eqref{incre1} also yields
\begin{equation}\label{gapcondition}
	\displaystyle{\lim_{|k|\rightarrow +\infty}| (k+1) a(k+1)- k a(k)|=+\infty},\;\;\text{where}\;k\;\text{runs over}\;  \mathbb{I}.
\end{equation}

Next we recall that with properties ($H1$), ($H2$), ($H3$), and \eqref{gapcondition}  we may apply the results in \cite[Theorem 1.3 and Remark 1.4]{Pastor and Vielma}  to conclude that the equation
$$
	\partial_{t}v=\partial_{x}\mathcal{A}u+Gh
$$
is exactly controllable in $H^s_p(\mathbb{T})$, for any $s\in \mathbb{R}$. More precisely, we obtain the following result.
\begin{thm}\label{ControlLa}
	Let  $\mu,s\in \mathbb{R}$ be given. 
	Then for any $T> 0$ and   for each $v_{0}, v_{1}\in H^{s}_p(\mathbb{T})$, with $\widehat{v_0}(0)=\widehat{v_1}(0)$, there exists a function $h\in L^{2}([0,T];H_p^{s}(\mathbb{T}))$
	such that the unique solution $v$  of the non-homogeneous system
	\begin{equation*}
		\begin{cases}
			v\in C([0,T];H^{s}(\mathbb{T})), \hbox{}\\
			\partial_{t}v= -\beta D^{2m}\partial_{x}v-\alpha \mathcal{H}^{2r}\partial_{x}v-2\mu \partial_{x}v
			+ Gh\in H_p^{s-(2m+1)}(\mathbb{T}),\;  t\in (0,T),\\
			v(0)=v_{0}\in H_{0}^{s}(\mathbb{T}), 
		\end{cases}
	\end{equation*}
	satisfies $v(T)=u_{1}.$ Furthermore,
	\begin{equation*}
		\|h\|_{L^{2}([0,T];H_p^{s}(\mathbb{T}))} \leq \nu\; (\|v_{0}\|_{H_p^{s}(\mathbb{T})}
		+\|v_{1}\|_{H_p^{s}(\mathbb{T})}),
	\end{equation*}
	for some positive constant $\nu$ depending on  $s,g,$ and $T.$
\end{thm}	

Regarding the stabilization property, the results in \cite[Theorem 1.8]{Pastor and Vielma} implies  the following.

\begin{thm}\label{estabilization}
	Let  $\mu,s\in \mathbb{R}$ be given. Then for any  $\lambda>0$ and $v_{0}\in H_{0}^{s}(\mathbb{T})$ there exists a bounded linear  operator $K_{\lambda}$ acting from $H_{0}^{s}(\mathbb{T})$ into itself such that the unique solution $v$ of the closed-loop system
	\begin{equation*}
		\begin{cases}
			v\in C([0,+\infty);H^{s}_{0}(\mathbb{T})), \hbox{}\\
			\partial_{t}v= - \beta D^{2m}\partial_{x}v-\alpha \mathcal{H}^{2r}\partial_{x}v-2\mu \partial_{x}v+ K_{\lambda}v\in H_{0}^{s-(2m+1)}(\mathbb{T}),  t> 0,\\
			v(0)=v_{0}\in H_{0}^{s}(\mathbb{T}), 
		\end{cases}
	\end{equation*}		
	satisfies
	\begin{equation}\label{decayrate}
	\|v(\cdot,t)\|_{H_{0}^{s}(\mathbb{T})}\leq
	Me^{-\lambda t}\|v_{0}\|_{H_{0}^{s}(\mathbb{T})},
\end{equation}
	for all $t\geq0,$ and some positive constant $M$ depending on $g$, $\lambda$ and $s$.
\end{thm}

\begin{rem}
According to \cite[Theorem 1.7]{Pastor and Vielma} (see also \cite{Slemrod}), the feedback law $-GG^{\ast}v,$ where $G^\ast$ denotes the adjoint operator of $G$, stabilizes the closed-loop system and \eqref{decayrate} holds for some $\lambda_0>0$.
\end{rem}

We finalize this section with the proof of a unique continuation property for the linear equation, which will be used in the proof of our main result. For this, we recall that, in view of ($H3$) there are only
finitely many integers in $\mathbb{I},$ say, $k_{j},$ $j=1,2,\cdots,n_{0}^{\ast},$ for some $n_{0}^{\ast}\in \mathbb{N},$
such that one can find another integer $k\neq k_{j}$
with $\lambda_{k}=\lambda_{k_{j}}$. By defining
$$\mathbb{I}_{j}:=\{k\in \mathbb{Z}: k\neq k_{j}, \lambda_{k}=\lambda_{k_{j}}\},\;\;\;\;j=1,2,\cdots,n_{0}^{\ast},
$$
we obtain the  pairwise disjoint union,
\begin{equation}\label{zdecomp}
	\mathbb{Z}=\mathbb{I}\cup\mathbb{I}_{1}\cup\mathbb{I}_{2}\cup\cdots\cup\mathbb{I}_{n_{0}^{\ast}}.
\end{equation}

Let  $H$ be the closure of ${\text{span}\{e^{-i\lambda_{k}t}: k\in \mathbb{I}\}}$ in $L^{2}([0,T]).$ It is not difficult to show (see Step 2 in the proof of Theorem 1.3 in \cite{Pastor and Vielma}) that
there exists a unique biorthogonal basis, say, $\{q_{j}\}_{j\in \mathbb{I}}\subseteq H$, to $\{e^{-i\lambda_{k}t}\}_{k\in \mathbb{I}}$, which gives
\begin{equation}\label{dualbg}
	(e^{-i\lambda_{k}t}\;,\;q_{j})_{H}=\int_{0}^{T}e^{-i\lambda_{k}t}\overline{q_{j}(t)}\;dt=\delta_{kj},\;\;\;k,j\in \mathbb{I},
\end{equation}
where $\delta_{kj}$ denotes the Kronecker delta. We define the sequence  $q_{j}$ for $j$ running on  $\mathbb{Z}$ as follows: In view of \eqref{zdecomp} we need to define this sequence for indices in $\mathbb{I}_j$, $j=1,\cdots, n_{0}^{\ast}$. But, from  $(H2)$ we see that $\mathbb{I}_{j}$ contains at most $4$ elements. Without loss of generality, we may assume that all multiple eigenvalues have multiplicity $5$ to write
\begin{equation*}
	\mathbb{I}_{j}=\{k_{j,1}, k_{j,2}, k_{j,3}, ,k_{j,4}\},\;\;\;\;j=1,2,\cdots,n_{0}^{\ast}.
\end{equation*}
In what follows we write $k_{j,0}$ for $k_j$. Given $k_{j,l}\in \mathbb{I}_{j}$ we define
$q_{k_{j,l}}$ to be $q_{k_{j,0}}$, so that $\lambda_{k_{j,l}}=\lambda_{k_{j}}$
for any $j=1,2,\cdots,n_{0}^{\ast}$ and $l=0,1,2,3,4.$ 

Now we are in a position to prove the following unique continuation property.
\begin{prop}\label{UCPLgBe}
	Let  $\mu\in \mathbb{R}$ be given. If $v\in C([0,T];H^s_{0}(\mathbb{T}))$, $s\in\mathbb{R}$, is a solution of
	\begin{equation}\label{UCPLgBe1}
		\left \{
		\begin{array}{l l}
			\partial_{t}v+\beta D^{2m}\partial_{x}v+\alpha \mathcal{H}^{2r}\partial_{x}v+2\mu \partial_{x}v
			=0,&  \;(x,t)\in \mathbb{T}\times (0,T),\\
			v(x,t)=0, &\text{a.e.} \;(x,t)\in(a,b)\times(0,T),
		\end{array}
		\right.
	\end{equation}	
	for some $T>0$ and $0\leq a<b\leq 2\pi,$ then $v(x,t)=0$ for almost every $(x,t)\in \mathbb{T}\times (0,T).$
\end{prop}
\begin{proof}
The proof is similar to that of Proposition 7 in \cite{Flores OH and Smith}; however, since we have the presence of multiples eigenvalues, a more careful analysis must be performed.	Let $v_{0}(x):=v(x,0)$. Because $v_0\in H^s_{0}(\mathbb{T})$, the unique solution of the differential  equation in  \eqref{UCPLgBe1} can be written as 
	$$v(x,t)=\sum_{l\in \mathbb{Z}} e^{i \lambda_{l}t}\widehat{v_{0}}(l)e^{ilx}.$$
The idea is to show that $\widehat{v_0}(l)=0$ for any $l\in\mathbb{Z}$.	First of all, note that for almost every $x\in (a,b)$ and any $n\in\mathbb{Z}$,
	\begin{equation}\label{UCPLgBe2}
			\begin{split}
				0&=\langle v(x,t), q_{n}(t) \rangle_{L^{2}([0,T])}\\
				&=\sum_{l\in \mathbb{Z}}e^{ilx} \widehat{v_{0}}(l)\int_{0}^{T} e^{i \lambda_{l}t} \overline{q_{n}(t)}dt\\
				&=\sum_{k\in \mathbb{Z}}e^{-ikx} \overline{\widehat{v_{0}}(k)}\int_{0}^{T} e^{-i \lambda_{k}t} \overline{q_{n}(t)}dt,
		\end{split}
	\end{equation}
	where we have performed the change of variable $l=-k$ and used the facts that $\lambda_{-k}=-\lambda_{k}$ and $\widehat{v_{0}}(-k)=\overline{\widehat{v_{0}}(k)}$.  Let us analyze  the right-hand side of \eqref{UCPLgBe2} according to $n$. We look for the sum in accordance with the decomposition 
	$$
	\mathbb{Z}=\widetilde{\mathbb{I}}\cup \{k_1,\ldots,k_{n_0^*}\}\cup\mathbb{I}_{1}\cup\mathbb{I}_{2}\cup\cdots\cup\mathbb{I}_{n_{0}^{\ast}},
	$$
	where $\widetilde{\mathbb{I}}:=
	\mathbb{I}-\{k_{1},\ldots,k_{n_{0}^{*}}\}.$
	
	 Assume first $n\in \widetilde{\mathbb{I}}$.  
	Note, if $k$ runs over $\widetilde{\mathbb{I}}$, from \eqref{dualbg}, all terms vanish except the one with index $n$. Also, if $k\in\{k_1,\ldots,k_{n_0^*}\}$ the integral term in \eqref{UCPLgBe2} is zero in view of \eqref{dualbg}. Finally, if $k\in \mathbb{I}_j$ for some $j=1,\ldots,n_0^*$ then $k=k_{j,l}$ for some $l=1,\ldots,4$ and in this case $\lambda_{k_{j,l}}=\lambda_{k_{j,0}}=\lambda_{k_{j}}$ with $k_j\in \{k_{1},\ldots,k_{n_{0}^{*}}\}$. Thus,
	$$
	\int_{0}^{T} e^{-i \lambda_{k}t} \overline{q_{n}(t)}dt=\int_{0}^{T} e^{-i \lambda_{k_j}t} \overline{q_{n}(t)}dt=0,
	$$
	because $n\neq k_j$. Hence, from \eqref{UCPLgBe2} we obtain
		\begin{equation*}
		0=e^{-inx}\overline{\widehat{v_{0}}(n)},\;\text{for almost every}\;x\in(a,b),
	\end{equation*}
	and consequently $\widehat{v_{0}}(n)=0,$ for all $n\in \widetilde{\mathbb{I}}$.
	
	On the other hand, if $n\in\mathbb{I}_{j}\cup\{k_{j}\}$  for some $j\in\{1, \ldots,n_0^*\}$ then $n=k_{j,l_0}$ for some $l_0\in\{0,1,2,3,4\}$. Since $\lambda_{k_{j,l_0}}=\lambda_{k_j}$, a similar analysis as above implies the integral in \eqref{UCPLgBe2} is zero, except for those indices $k$ in  $\mathbb{I}_{j}\cup\{k_{j}\}$. In particular, \eqref{UCPLgBe2} reduces to
	\begin{equation*}
		\displaystyle{0=\sum_{l_{0}=0}^{4} e^{-ik_{j,l_{0}}x}
			\overline{\widehat{v_{0}}(k_{j,l_{0}})}},\;\text{for almost every}\;x\in(a,b).
	\end{equation*}
	Since, all elements in the set $\left\{ k_{j,l_{0}}\in \mathbb{Z}: l_{0}=0,1,2,3,4\right\}$ are distint and the functions $$\left\{ e^{-ik_{j,l_{0}}x}: l_{0}=0,1,2,3,4\right\}$$ are linearly independent we conclude that $\widehat{v_{0}}(k_{j,l_{0}})=0,$ for such
	$j\in\{1, \ldots,n_0^*\}$ and all $l_{0}\in\{0,1,2,3,4\}$. Hence, taking the index $n$ over sets of the form  $\mathbb{I}_{j}\cup\{k_{j}\}$ for $j=1,2\cdots,n_{0}^{\ast}$ one can show that $\widehat{v_{0}}(k_{j,l_{0}})=0,$ for all
	$j\in\{1, \ldots,n_0^*\}$ and  $l_{0}\in\{0,1,2,3,4\}$, which gives that $\widehat{v}(l)=0,$ for all $l\in\mathbb{Z}-\widetilde{\mathbb{I}}.$ Therefore $v(x,t)=0$ for almost every $(x,t)\in \mathbb{T}\times (0,T)$ and the proof of the proposition is completed.
\end{proof}

\section{Proof of Theorem \ref{NLcontrol}} \label{Proof}
In this section we give the proof of Theorem  \ref{NLcontrol}, which is similar to that of \cite[Theorem 1.1]{10} (see also \cite[Proposition 1.1 ]{1}).

\begin{proof}[Proof of Theorem \ref{NLcontrol}]
Assume $u_{0},u_{1}\in H_{p}^{s+1}(\mathbb{T})$ satisfy $[u_{0}]=[u_{1}].$ From Theorem \ref{ControlLa} with $\mu \equiv 0$ there exists $h_{1}\in L^{2}(0,T; H_{p}^{s+1}(\mathbb{T}))$ such that the solution
	$u\in C([0,T];H^{s+1}_{p}(\mathbb{T}))$ of the linear IVP
	\begin{equation*}\label{con1}
		\begin{cases}
			\partial_{t}u= -\beta D^{2m}\partial_{x}u-\alpha \mathcal{H}^{2r}\partial_{x}u
			+ Gh_{1}\in H_{p}^{s+1-(2m+1)}(\mathbb{T}),\;  t\in (0,T),\\
			u(0)=u_{0}, 
		\end{cases}
	\end{equation*}
	fulfills $u(T)=u_{1}.$ Note that the nonlinear term $\partial_{x}(u^2)$ belongs to the space $L^{2}(0,T; H_{p}^{s}(\mathbb{T})).$ Adding this term to both sides of the above equation, one obtains
	\begin{equation*}\label{con2}
		\begin{cases}
			\partial_{t}u +\beta D^{2m}\partial_{x}u+\alpha \mathcal{H}^{2r}\partial_{x}u +\partial_{x}(u^2)=
			G(h_{1})  +\partial_{x}(u^2),\;\; x\in\mathbb{T},\; t\in (0,T),\\
			u(0)=u_{0}, \;\;u(T)=u_{1}.
		\end{cases}
	\end{equation*}
	Thus, it is enough to show that there exists  $h_{2}\in L^{2}(0,T; H_{p}^{s}(\mathbb{T}))$ such that
	\begin{equation}\label{con3}
		\partial_{x}(u^2)(x,t)=G(h_{2})(x,t),\;\text{for all}\;x\in\mathbb{T},\;\text{and}\;t\in (0,T).
	\end{equation}
By the definition of operator $G$ (see \eqref{Gdef} and \eqref{NLcontro2}), equation \eqref{con3} reduces to
	\begin{equation}\label{con4}
			2\pi \partial_{x}(u^2)(x,t)=h_{2}(x,t)-\frac{1}{2\pi}\int_{\mathbb{T}}h_{2}(y,t)dy.
	\end{equation}

	Next, we will show the existence of $h_{2}$.	We start by defining the map $\tilde{G}:H_p^{s}(\mathbb{T})\longrightarrow H_p^{s}(\mathbb{T})$ as
	$$\tilde{G}v(x)=v(x)-\frac{1}{2\pi}\int_{\mathbb{T}}v(y)dy.$$
It is not difficult to check that  $\tilde{G}$ is a self-adjoint bounded linear operator. In addition,
	\begin{equation}\label{Kerdef}
		\text{Ker}(\tilde{G}):=\left\{v\in H_p^{s}(\mathbb{T}): \;\tilde{G}(v)=0\right\}=\text{span}\{1\}
	\end{equation}
	and
	$$
	\text{R}(I-\tilde{G}):=\left\{u\in H_p^{s}(\mathbb{T}): \; u=(I-\tilde{G})v,\; \mbox{for some}\; v\in H_p^{s}(\mathbb{T}) \right\}=\text{span}\{1\}.
	$$
Since $\text{dim}(\text{R}(I-\tilde{G}))$ is finite, we have  that $I-\tilde{G}$ is compact (see, for instance, \cite[page 157]{Brezis}). Therefore, $\tilde{G}$ is a Fredholm operator of index zero (see, for instance, \cite[page 168]{Brezis}), which implies that $\text{R}(\tilde{G})$ is closed in $H_p^{s}(\mathbb{T})$.
	Thus, the Fredholm alternative (see \cite[Theorem 6.6]{Brezis}) implies that
	\begin{equation}\label{con6}
		\overline{\text{R}(\tilde{G})}=\text{R}(\tilde{G})= \text{R}(I-(I-\tilde{G}))=\text{Ker}(I-(I-\tilde{G})^{\ast})^{\bot}
		=\text{Ker}(\tilde{G}^{\ast})^{\bot}
		=\text{Ker}(\tilde{G})^{\bot}.
	\end{equation}
	In view of \eqref{Kerdef} and \eqref{con6} it follows that $\tilde{G}:\text{Ker}(\tilde{G})^{\bot}\to \text{Ker}(\tilde{G})^{\bot}$ is invertible with a bounded inverse.
	
Note that for any $t\in (0,T),$
\[
		\begin{split}
			\left(2\pi \partial_{x}(u^2),1\right)_{H_p^{s}(\mathbb{T})}
			=2\pi  \widehat{\partial_{x}(u^{2})}(0)
			=2\pi \int_{\mathbb{T}}\partial_{x}(u^{2})dx
			=0,
	\end{split}\]
	which yields 
	$$2\pi \partial_{x}(u^2)\in \text{Ker}(\tilde{G})^{\bot}=\text{R}(\tilde{G}),\;\;\text{for any}\;t\in (0,T).$$
Therefore, equation \eqref{con4} has a solution $h_{2}(\cdot,t)\in H_{p}^{s}(\mathbb{T}),$ for each fixed $t\in (0,T).$ Since $\tilde{G}^{-1}$ is  bounded, we obtain that
	$h_{2}\in L^{2}(0,T;H_{p}^{s}(\mathbb{T})).$
	Hence, considering the control function $h=h_{1}+h_{2},$ we complete the proof of the theorem.
\end{proof}	

\section{Linear and nonlinear estimates}\label{Bourgspace}
In this section we introduce the Fourier transform  restriction spaces, the so-called dissipation-normalized Bourgain's space and derive some preliminary linear and nonlinear estimates in order to prove our global well-posedness and stabilization results for the DGB equation.

\subsection{Dissipation-normalized Bourgain's space}\label{Bourgspace1}

For given $b\in \mathbb{R}$ and $\delta>0$ we define the dissipation-normalized Bourgain's space $Z^{b}$ associated to the DGB equation \eqref{GB6} on $\mathbb{T}$ as the closure of the Schwartz space $\mathcal{S}(\mathbb{T}\times \mathbb{R})$ under the norm
\begin{equation*}
\displaystyle{\|v\|_{Z^{b}}}:=	
\begin{cases}
\displaystyle{\left(\sum_{k\in \mathbb{Z}^{\ast}}\int_{\mathbb{R}}\langle k \rangle^{2b\delta} \left\langle\frac{\tau-\lambda_{k}} {\langle k\rangle^{\delta}} \right\rangle^{2b} |\widehat{v}(k,\tau)|^{2} d\tau\right)^{\frac{1}{2}},\;\;\text{if} \; b\in \left(-\frac{1}{2},\frac{1}{2}\right)},\\
\displaystyle{ \left(\sum_{k\in \mathbb{Z}^{\ast}}\int_{\mathbb{R}}\langle k \rangle^{\text{sgn}(b)\delta} \left\langle\frac{\tau-\lambda_{k}} {\langle k\rangle^{\delta}} \right\rangle^{2b} |\widehat{v}(k,\tau)|^{2} d\tau\right)^{\frac{1}{2}},\;\;\text{if} \; b\notin \left(-\frac{1}{2},\frac{1}{2}\right)}	,
\end{cases}
\end{equation*}
where  $ \mathbb{Z}^{\ast}:=\mathbb{Z}-\{0\},$ $\langle \cdot \rangle:=(1+|\cdot|^{2})^{\frac{1}{2}}$,
\begin{equation}\label{symboldef}
	\lambda_{k}:=-\beta k|k|^{2m}+ \alpha k |k|^{2r}-2\mu k, \quad k\in\mathbb{Z},
\end{equation}
 and $\widehat{v}(k,\tau)$ denotes the Fourier transform of $v$ with respect to both space and time variables, namely,
$$\displaystyle{\widehat{v}(k,\tau):=\frac{1}{2\pi}\int_{\mathbb{R}}\int_{\mathbb{T}}v(x,t)e^{-i(t\tau+kx)}dx dt}.$$
Sometimes we use $\widehat{v}(k,t)$ (respectively $\widehat{v}(x,\tau)$) to denote the Fourier transform in space variable $x$ (respectively in time variable $t$).
It is easy to verify that $Z^{b}$ is  a Hilbert space with the natural inner product.

For a given interval $I\subset\mathbb{R}$, we define $Z^{b}(I)$ to be the restriction of $Z^{b}$ to the interval $I$ with norm 
\begin{equation*}
	\|f\|_{Z^{b}(I)}:=\inf\left\{\|\widetilde{f}\|_{Z^{b}}: \widetilde{f}=f\;\text{on}\;\mathbb{T}\times I\right\}.
\end{equation*}
If $I=[0,T],$ for simplicity, we denote $Z^{b}(I)$ by $Z^{b}_{T}.$

\begin{rem}\label{Bnorm3}
As we already said,	motivated by the usual Bourgain spaces (see \cite{13}), the authors in
	\cite{Flores OH and Smith} introduced this Bourgain-type weighted space to show that the  dispersion generalized BO equation \eqref{GBO} is stabilizable in $L^{2}_{0}(\mathbb{T}).$ The advantage that the  spaces $Z^{b}$ offers is that smoothing can be gained from both dissipation and dispersion simultaneously as we will se below.
\end{rem}

In what follows we recall some properties of the spaces $Z^{b}.$ The ideas to prove them are similar to those derived for the usual Bourgain spaces (see \cite[Section 3]{Flores OH and Smith} and \cite[Section 2.6]{11}).

\begin{prop}[Properties of Bourgain's spaces]\label{prop2}
	Let $I\subset\mathbb{R}$ be an interval and $\delta,b,b'\in\mathbb{R}$ with $\delta>0$. 
	\begin{itemize}
		\item[(i)] If $b'\leq b$ then $Z^{b}$ (resp. $Z^{b}(I)$) is continuously embedded in $Z^{b'}$ (resp. $Z^{b'}(I)$).
		\item[(ii)] The space $Z^{b}$ is reflexive and its dual is given by $Z^{-b}.$
		\item[(iii)] If  $b>\frac{1}{2}$ then $Z^{b}$ (resp. $Z^{b}(I)$) is continuously embedded in the space $C(\mathbb{R};L^{2}_{0}(\mathbb{T})\cap L^{2}(\mathbb{R};H_{0}^{\frac{\delta}{2}}(\mathbb{T}))$ (resp. $C(I;L^{2}_{0}(\mathbb{T})\cap L^{2}([0,T];H_{0}^{\frac{\delta}{2}}(\mathbb{T}))$). Furthermore, there exists a positive constant $C$ depending only on $b$ such that
		$$\|v\|_{C(\mathbb{R};L^{2}_{0}(\mathbb{T}))} \leq C\|v\|_{Z^{b}}.$$
	\end{itemize}
\end{prop}
\begin{proof}
	Part (i) is clear. Part (ii) follows from the fact that the symbol $\lambda_{k}$ in \eqref{symboldef} is an odd function of $k$ (see \cite[Section 2.6]{11}). Finally, part (iii) follows from the definition of the norm in $Z^{b}$ and an argument similar to the one in the proof of Proposition 1 in \cite{Flores OH and Smith}, so we omit the details.
\end{proof}

\subsection{Linear and integral estimates}\label{linint}
To derive some key linear and integral estimates we follow a similar approach as in 
\cite[Section 2.3]{Flores OH and Smith}. Let us start by decomposing
the localized damping $GD^{\delta}G$  ($\delta>0$)  present in 
equation \eqref{GB6} as follows: using the definition of $G$ we write
$$
GD^{\delta}Gv=gD^{\delta}(gv)-\frac{1}{2\pi}\int_{\mathbb{T}}g(y)D^{\delta}(gv)(y)dy+Rv=:\mathcal{B}v+Rv,
$$
where
\begin{equation}\label{locdamp3}
	{
		\begin{split}
			(Rv)(x)&:=\frac{1}{2\pi}\int_{\mathbb{T}}g(y) D^{\delta}(gv)(y)dy-\left(\int_{\mathbb{T}}g(y)v(y)dy\right) g(x)D^{\delta}(g)(x)\\
			&\qquad -\left(\int_{\mathbb{T}} g(y)D^{\delta}(gv)(y)dy\right) g(x) + \left( \int_{\mathbb{T}}g(z)v(z)dz\right) \left(\int_{\mathbb{T}}g(y) D^{\delta}(g)(y)dy \right) g(x).
	\end{split}}
\end{equation}
Note that $\widehat{\mathcal{B}v}(0)=0$; so the contribution of this term occurs only for $k\neq0$. In addition, using the definition of $D^\delta$ we may decompose
$$
\mathcal{B}v=\widetilde{D}^{\delta}v+N_{1}v,
$$
where $\widetilde{D}^{\delta}$ and $N_1$ are Fourier multipliers with symbols given, respectively, by
\[
	d(k):=
\begin{cases}
{\displaystyle \sum_{l\in \mathbb{Z}}|l|^{\delta}|\widehat{g}(l-k)|^{2}, \quad k\in\mathbb{Z}^{\ast}},\\
0, \qquad k=0,
\end{cases}
\]
and 
\[
c(k):=
\begin{cases}
{\displaystyle\sum_{l\in \mathbb{Z}} \sum_{\substack{n\in \mathbb{Z}\\
		n\neq k}} |l|^{\delta}\widehat{g}(k-l)\widehat{g}(l-n) \widehat{v}(n), \quad k\in\mathbb{Z}^{\ast}},\\
	0, \qquad k=0.
\end{cases}
\]
Thus, we may write
\begin{equation}\label{locdamp}
	GD^{\delta}Gv=\widetilde{D}^{\delta}v+N_{1}v+Rv.
\end{equation}

Some properties of the above operators are given next. The first result says that operator $ \widetilde{D}^{\delta} $ behaves like a derivative of order $\delta.$
\begin{lem}\label{locdamp4}
	Let $\delta>0$ be given and assume that $g$ satisfies \eqref{OPG1}.
	For any $k\in \mathbb{Z}^{\ast}$ there exist $c=c(\delta)$ and $C=C(\delta,g)$ (uniform constants) such that
	$\displaystyle{c \langle k \rangle^{\delta}  \leq {d}(k)\leq   C \langle k \rangle^{\delta}},$
	that is, $$\displaystyle{{d}(k)\sim_{\delta,g} \langle k \rangle^{\delta}}.$$
\end{lem}
\begin{proof}
	See proof of Claim 1 in \cite{Flores OH and Smith}.
\end{proof}

\begin{lem}\label{locdamp5}
	Let $\delta>0$ be given and $g$ as in \eqref{OPG1}.	The linear operator $R$ defined in \eqref{locdamp3} is bounded from $L^{2}_{p}(\mathbb{T})$ into itself. Furthermore, there exists a positive constant $C=C(\delta,g)$ such that
	$$\|Rv\|_{L^{2}_{p}(\mathbb{T})}\leq C \|v\|_{L^{2}_{p}(\mathbb{T})},$$
	for all $v\in L^{2}_{p}(\mathbb{T}).$
\end{lem}
\begin{proof}
The proof follows as an application of Cauchy-Schwarz's inequality and Parseval's identity.
\end{proof}

In view of the decomposition \eqref{locdamp}, system \eqref{GB6} can be rewritten in the form
\begin{equation}\label{GB7}
	\begin{cases}
		\partial_{t}v+\beta D^{2m}\partial_{x}v+\alpha \mathcal{H}^{2r}\partial_{x}v+2\mu  \partial_{x}v+\widetilde{D}^{\delta}v=- \partial_{x}(v^{2})-N_{1}v-Rv, \quad x\in \mathbb{T},\; t>0,\\
		v(x,0)=v_{0}, 
	\end{cases}
\end{equation}
so that the terms $N_{1}v$ and $Rv$ will be treated as nonlinear ones. Next, we  rewrite  \eqref{GB7} in its equivalent integral formulation, namely,
\begin{equation*}
	v(t)=S_{\mu}(t)v_{0}-\int_{0}^{t}S_{\mu}(t-s)\left( \partial_{x}(v^{2})+N_{1}v+Rv\right)(s)ds,
\end{equation*}
where, for any $t\in \mathbb{R},$ $S_{\mu}(t)$ is defined for any $v_{0}\in L^{2}_{p}(\mathbb{T})$ by
\begin{equation}\label{GB9}
	S_{\mu}(t)v_{0}:=e^{(-\beta D^{2m}\partial_{x}- \alpha \mathcal{H}^{2r}\partial_{x}-2\mu \partial_{x})t-\widetilde{D}^{\delta}|t|}v_{0}
	=\left(e^{i\lambda_{k}t-{d}(k)|t|}\widehat{v_{0}}(k)\right)^{\vee}.
\end{equation}

It must be clear that to prevent a backward parabolic propagation, the absolute value was placed around time variables associated with the dissipative coefficients ${d}(k)$ in \eqref{GB9}. It is easy to show that the family of operators $\{S_{\mu}(t)\}_{t\geq 0}$ defines a strongly continuous one-parameter semigroup of contractions on $L^{2}_{p}(\mathbb{T}).$ Its infinitesimal generator $\widetilde{\mathcal{A}}: \mathcal{D}(\widetilde{\mathcal{A}}) \subset L^{2}_{p}(\mathbb{T}) \to L^{2}_{p}(\mathbb{T})$
defined by $\widetilde{\mathcal{A}} \varphi:= -\beta D^{2m}\partial_{x}\varphi- \alpha \mathcal{H}^{2r}\partial_{x}\varphi-2\mu \partial_{x}\varphi-\widetilde{D}^{\delta}\varphi$ has domain $\mathcal{D}(\widetilde{\mathcal{A}})=H^{s_{0}}_{p}(\mathbb{T}),$ with $s_{0}=\max\{2m+1,\delta\}.$

To derive some  estimates localized in time variable, we introduce a cut-off function
$\eta\in C^{\infty}_{c}(\mathbb{R})$  such that $\eta\equiv1,$ if $t\in[-1,1]$ and $\eta\equiv0$,\;if $t \notin (-2,2).$
For $T>0$ given, we define $\eta_T \in C^{\infty}_{c}(\mathbb{R})$ by $$\eta_T(t):=\eta\left(\frac{t}{T}\right).$$

The ideas to prove the following three results are similar to those derived in the dispersion generalized Benjamin-Ono case.

\begin{prop}\label{locdamp6}
	Let $\delta>0,$ $b<\frac{3}{2}$ and $g$ as in \eqref{OPG1}. Then for any $v_{0}\in L^{2}_{0}(\mathbb{T})$	 there exists a positive constant $C=C(\delta,b,g)$ such that
	$$\|S_{\mu}(t)(v_{0})\|_{Z^{b}}\leq C \|v_{0}\|_{L^{2}_{0}(\mathbb{T})}.$$
\end{prop}
\begin{proof}
	See \cite[Proposition 2]{Flores OH and Smith}.
\end{proof}

Next proposition shows that the space $Z^{b}$ inherits a special property of Bourgain's spaces. In what well-posedness is concerned, such a property  is useful in the large data theory, as it allows to keep certain $Z^{b}$ norms of a solution small by localizing to a sufficiently small time interval.
\begin{prop}\label{locdamp7}
	Let  $\delta>0$ and $b\in \mathbb{R}$ be given. Then for any $u\in Z^{b}$ there exists a positive constant $C=C(\eta,b)$ such that
	$$\|\eta(t)u\|_{Z^{b}}\leq C \|u\|_{Z^{b}}.$$
	Furthermore, given $0<T<1$ and $-\frac{1}{2}<b'\leq b <\frac{1}{2},$ there exists $C=C(\eta,b,b')>0$ such that
	$$\|\eta_{T}(t)u\|_{Z^{b'}}\leq C T^{b-b'}\|u\|_{Z^{b}}.$$
\end{prop}
\begin{proof}
See \cite[Proposition 3]{Flores OH and Smith}.
\end{proof}

\begin{thm}\label{locdamp8}
	Let $\delta>0$ be given and $g$ as in \eqref{OPG1}. Then for any $b\in\left(\frac{1}{2},\frac{3}{2}\right)$ and any Schwartz function $f$ there exists a positive constant $C=C(\delta,b,g)$ such that
	\begin{equation}\label{intsmooteff}
		\left\|\int_{0}^{t}S_{\mu}(t-s)f(s)ds\right\|_{Z^{b}}\leq C\left\|D^{-\delta\left(b-\frac{1}{2}\right)} f\right\|_{Z^{b-1}}.
	\end{equation}
\end{thm}
\begin{proof}
See \cite[Proposition 4]{Flores OH and Smith}.
\end{proof}

The proofs of the previous two results use some properties of the $A_{p}-$weights theory, we refer the reader to \cite[Chapter 9]{Grafakos} for additional details on this issue.
The introduction of the dissipation-normalized Bourgain spaces is directly reflected in Theorem \ref{locdamp8}, since it reveals a smoothing effect of order $\delta/2$ for $b\sim 1/2$. Indeed, in this case, the norm on the right hand side of \eqref{intsmooteff} is approximately $\|f\|_{Z^{-\frac{1}{2}}}$, which by definition, may be bounded by  $ \|f\|_{L^{2}_{t}H_{x}^{-\frac{\delta}{2}}}.$  As pointed out in \cite{Flores OH and Smith}, this smoothing effect is analogous to the $\frac{1}{2}$ derivative gain achieved for the proof of stabilization of BO equation using the operator $GD^{1}G$ in \cite[Proposition 2.16]{Linares Rosier}.

In what follows, given two integers $n$ and $k$, we will write $k\sim n$ provided $|k|\sim |n|$.
The following result is fundamental to establish an adequate estimate for the linear operator $N_{1}$.
\begin{lem}\label{stiN}
	Let $\lambda_{k}$ defined as in \eqref{symboldef} and assume $\delta>0.$ Then for any $k,n\in \mathbb{Z}$ with $k\neq n$ and $k\sim n$,

	\begin{equation}\label{incre6}
		\max\left\{ \left\langle \frac{\tau-\lambda_{k}}{\langle k \rangle^{\delta} } \right\rangle, \left\langle \frac{\tau-\lambda_{n}}{\langle n \rangle^{\delta} } \right\rangle \right \} 
		\gtrsim_{\beta, \delta, m}
		\max\left\{ \langle k \rangle, \langle n \rangle \right\}^{2m-\delta},
	\end{equation}
provide $|k|$ is sufficiently large.
\end{lem}
\begin{proof}
By symmetry, without loss of generality, we may assume $|n|>|k|$. 	Using   the identity
	$\displaystyle{\max\{c_{1},c_{2}\}= \frac{1}{2}\left(c_{1}+c_{2}+|c_{1}-c_{2}|\right)}$ 
	and the fact that $\langle x\rangle\sim 1+|x|$
	we get
	\begin{equation}\label{incre7}
		{
			\begin{split}
				M&:=\max\left\{ \left\langle \frac{\tau-\lambda_{k}}{\langle k \rangle^{\delta} } \right\rangle, \left\langle \frac{\tau-\lambda_{n}}{\langle n \rangle^{\delta} } \right\rangle \right \}\\
				&\gtrsim \left| \frac{\tau-\lambda_{k}}{\langle k \rangle^{\delta} } \right| + \left| \frac{\tau-\lambda_{n}}{\langle n \rangle^{\delta} } \right|
				+ \left| \left| \frac{\tau-\lambda_{k}}{\langle k \rangle^{\delta} } \right|- \left| \frac{\tau-\lambda_{n}}{\langle n \rangle^{\delta} } \right| \right|.
		\end{split}}
	\end{equation}

Recall that in Section \ref{Linear Systems} we have proved that the sequence $\{\lambda_{k}\}$ is strictly decreasing for $|k|$ sufficiently large; in addition, $\lambda_{k}>0$ for $k<0$ (large) and $\lambda_{k}<0$ for $k>0$ (large).
We now split our analysis into two cases.\\

\noindent	\textbf{Case 1.} $n>k>0$ or $n<k<0$. 
	
	Here we will consider only that $k$ and $n$ satisfy  $n>k>0$ because the case $n<k<0$ can be treated in a similar fashion. Hence, we have $0>\lambda_{k}>\lambda_{n}$ and $\tau-\lambda_{k}<\tau -\lambda_{n}.$ We now split this case into three subcases.
	
\noindent	\textbf{Subcase 1.1.} $0<\tau-\lambda_{k}.$ 

From \eqref{incre7} we obtain
\begin{equation}\label{incre8}
{
\begin{split}
M&\gtrsim  \frac{\tau-\lambda_{k}}{\langle k \rangle^{\delta} }  +  \frac{\tau-\lambda_{n}}{\langle n \rangle^{\delta} }\\
&=\frac{\tau-\lambda_{k}}{\langle k \rangle^{\delta} } + \frac{\tau-\lambda_{k}}{\langle n \rangle^{\delta} } - \frac{\tau-\lambda_{k}}{\langle n \rangle^{\delta} } +  \frac{\tau-\lambda_{n}}{\langle n \rangle^{\delta} }\\
&=(\tau-\lambda_{k})
\left(\frac{1}{\langle k \rangle^{\delta}} + \frac{1}{\langle n \rangle^{\delta}} \right) +  \frac{\lambda_{k}-\lambda_{n}}{\langle n \rangle^{\delta} }\\
& \geq \frac{\lambda_{k}-\lambda_{n}}{\langle n \rangle^{\delta} }.
\end{split}}
\end{equation}
	It follows from the Mean Value Theorem  that there exists $n^{\ast}\in \mathbb{R}$ with $n>n^{\ast}>k$ such that
	\begin{equation}\label{incre9}
		{
			\begin{split}
				\frac{\lambda_{k}-\lambda_{n}}{\langle n \rangle^{\delta} }
				&=\frac{\Big(\beta (2m+1) (n^{\ast})^{2m}-\alpha (2r+1) (n^{\ast})^{2r} +2 \mu\Big)(n-k)  }
				{\langle n \rangle^{\delta}}\\
				&\geq\frac{\beta (2m+1) (n^{\ast})^{2m}}
				{\langle n \rangle^{\delta}}\left( 1- \frac{\alpha (2r+1)}{\beta (2m+1)}(n^{\ast})^{2(r-m)} + \frac{2 \mu}{ \beta (2m+1) (n^{\ast})^{2m}} \right)	  
				\\
				&\geq \frac{\beta (2m+1) (n^{\ast})^{2m}}
				{2 \langle n \rangle^{\delta}},
		\end{split}}
	\end{equation}
	where we have used that $n-k\geq1$ and the fact that the term between the parenthesis approaches 1 for $k$ large enough. From \eqref{incre8}, \eqref{incre9} and the fact that $n \sim k$ we infer that
	$$M \gtrsim_{m,\beta}  \frac{ n^{2m}}
	{\langle n \rangle^{\delta}} \gtrsim_{m,\beta} \langle n \rangle^{2m-\delta},$$
	which is the desired inequality.

\noindent	\textbf{Subcase 1.2.} $\tau-\lambda_{k}<0<\tau-\lambda_{n}.$ 

Here, from \eqref{incre7} we obtain
	\begin{equation}\label{incre11}
		{
			\begin{split}
				M
				&=-\frac{\tau-\lambda_{k}}{\langle k \rangle^{\delta} } + \frac{\tau-\lambda_{k}}{\langle n \rangle^{\delta} } - \frac{\tau-\lambda_{k}}{\langle n \rangle^{\delta} } +  \frac{\tau-\lambda_{n}}{\langle n \rangle^{\delta} }\\
				&=-(\tau-\lambda_{k})
				\left(\frac{1}{\langle k \rangle^{\delta}} - \frac{1}{\langle n \rangle^{\delta}} \right) +  \frac{\lambda_{k}-\lambda_{n}}{\langle n \rangle^{\delta} }\\
				& \geq \frac{\lambda_{k}-\lambda_{n}}{\langle n \rangle^{\delta} }\\
				& \gtrsim_{m,\beta} \langle n \rangle^{2m-\delta},
		\end{split}}
	\end{equation}
	where we have used a similar estimate as in \eqref{incre9} in the last inequality.
	
\noindent	\textbf{Subcase 1.3.} $\tau-\lambda_{k}<\tau-\lambda_{n}<0.$ 

By using the second term in the sum \eqref{incre7} we deduce
	\begin{equation*}
		{
			\begin{split}
				M&\gtrsim 
				\left| - \frac{\tau-\lambda_{k}}{\langle k \rangle^{\delta} } +  \frac{\tau-\lambda_{n}}{\langle n \rangle^{\delta} }  \right|
				& \geq \frac{\lambda_{k}-\lambda_{n}}{\langle n \rangle^{\delta} }
				& \gtrsim_{m,\beta} \langle n \rangle^{2m-\delta}.
		\end{split}}
	\end{equation*}
	In all three subcases, \eqref{incre6} holds.\\

\noindent	\textbf{Case 2.} $n>0>k$ or  $k>0>n$.
	
We assume that $k$ and $n$ satisfy  $n>0>k$ (the case $k>0>n$ can be treated similarly). Thus we have $\lambda_{k}>0>\lambda_{n},$ $ \lambda_{k}-\lambda_{n}>\lambda_{k}>0,$ and $\tau-\lambda_{k}<\tau -\lambda_{n}.$  We also split this case into three other ones.
	
\noindent	\textbf{Subcase 2.1.} $0<\tau-\lambda_{k}.$ 

From \eqref{incre7} and similar computations as in \eqref{incre8}, we deduce
	\begin{equation}\label{incre17}
		{
			\begin{split}
				M\gtrsim  \frac{\tau-\lambda_{k}}{\langle k \rangle^{\delta} }  +  \frac{\tau-\lambda_{n}}{\langle n \rangle^{\delta} }
				 \geq \frac{\lambda_{k}-\lambda_{n}}{\langle n \rangle^{\delta} }
				 \geq \frac{\lambda_{k}}{\langle n \rangle^{\delta} }.
		\end{split}}
	\end{equation}
	Note that
	\begin{equation}\label{incre18}
		{
			\begin{split}
				\frac{\lambda_{k}}{\langle n \rangle^{\delta} }
				&=\frac{\beta |k|^{2m+1}-\alpha |k|^{2r+1}+ 2 \mu |k|}
				{\langle n \rangle^{\delta}}\\
				&=\frac{\beta |k|^{2m+1}}
				{\langle n \rangle^{\delta}} 
				\left( 1- \frac{\alpha}{\beta |k|^{2(m-r)}} + \frac{2 \mu}{ \beta  |k|^{2m}} \right)	  	\\
				&\geq \frac{\beta |k|^{2m+1}}
				{2 \langle n \rangle^{\delta}},
		\end{split}}
	\end{equation}
	where we have used the fact that the term between the parenthesis approaches 1 for $k$ sufficiently large. From \eqref{incre17}, \eqref{incre18} and the fact that $n \sim k$ we infer that
	$$M \gtrsim_{\beta}  \frac{ n^{2m}}
	{\langle n \rangle^{\delta}} \gtrsim_{\beta} \langle n \rangle^{2m-\delta}.$$

\noindent	\textbf{Subcase 2.2.} $\tau-\lambda_{k}<0<\tau-\lambda_{n}.$ 

Using \eqref{incre7} and similar computations as in \eqref{incre11} and \eqref{incre18}, we obtain
	\begin{equation*}
		{
			\begin{split}
				M\gtrsim  \frac{\lambda_{k}-\lambda_{n}}{\langle n \rangle^{\delta} }
				 \gtrsim_{m,\beta} \langle n \rangle^{2m-\delta}.
		\end{split}}
	\end{equation*}

\noindent	\textbf{Subcase 2.3.} $\tau-\lambda_{k}<\tau-\lambda_{n}<0.$

As in Subcase 1.3 we deduce
	\begin{equation*}
		{
			\begin{split}
				M\gtrsim 
				\left| - \frac{\tau-\lambda_{k}}{\langle k \rangle^{\delta} } +  \frac{\tau-\lambda_{n}}{\langle n \rangle^{\delta} }  \right|
				 \geq \frac{\lambda_{k}-\lambda_{n}}{\langle n \rangle^{\delta} }
				 \gtrsim_{m,\beta} \langle n \rangle^{2m-\delta}.
		\end{split}}
	\end{equation*}
Again	in all three subcases \eqref{incre6} holds. This completes the proof of the lemma.
\end{proof}

Next two results bring the estimate for the linear operators $N_1$ and $R$ on the dispersion-normalized Bourgain's space $Z^{b}_{T}.$

\begin{thm}\label{stiN1}
	Let $m,\delta \in \mathbb{R}$ be given with $2m>\delta>0,$ $m>\frac{1}{2},$ and $g$ as in \eqref{OPG1}. For any given $0<T<1$ and any $\displaystyle{\frac{1}{2}<b< \frac{2m}{2m+\delta}},$ assume $v\in Z_{T}^{b}$. Then,  there exists $\epsilon>0$ small such that
	$$\displaystyle{\|D^{-\delta \left(b-\frac{1}{2}\right) }  N_{1}(v)\|_{Z_{T}^{b-1}} \lesssim_{\epsilon,b,\delta,m, g} T^{\epsilon} \|v\|_{Z^{b}_{T}}.}$$
\end{thm}
\begin{proof}
	Let $u\in Z^{b}$ be such that $u(t)\equiv v(t)$ on $[0,T]$ and $\|u\|_{Z^{b}}\leq 2 \|v\|_{Z^{b}_{T}}.$
	From the hypothesis we have $\displaystyle{\frac{1}{2}<b<1}.$ Let $\epsilon>0$ to be chosen later. Assuming $\epsilon<\frac{1}{2}$ we obtain $\displaystyle{-\frac{1}{2}<b-1<b-1+\epsilon <\frac{1}{2}}$ and from Proposition \ref{locdamp7}
	we infer
	\begin{equation*}
		{\begin{split}
				\|D^{-\delta \left(b-\frac{1}{2}\right) }  N_{1}(v)\|_{Z_{T}^{b-1}}	
				& \leq  \|D^{-\delta \left(b-\frac{1}{2}\right) } \eta_{T}(t) N_{1}(u)\|_{Z^{b-1}}\\
				&\lesssim_{\eta,b,\epsilon}T^{\epsilon} \|D^{-\delta \left(b-\frac{1}{2}\right) }  N_{1}(u)\|_{Z^{b-1+\epsilon}}.
		\end{split}}
	\end{equation*}
	Now we estimate the term in the right-hand side  of the last inequality.
First note that for any $k\in \mathbb{Z}^{\ast},$ 
	\begin{equation}\label{stiN-1}
		\left\langle k \right\rangle^{2(b-1+\epsilon)\delta} |k|^{-\delta\left(b-\frac{1}{2}\right)2} \lesssim_{\delta,b} |k|^{-2\left(\frac{1}{2}-\epsilon\right)\delta}.
	\end{equation} 
	Hence, using the definition of the space $Z^{b-1+\epsilon}$ and \eqref{stiN-1} we obtain
	$$
	{
		\begin{split}
			\left\|D^{-\delta \left(b-\frac{1}{2}\right) }  N_{1}(u)\right\|_{Z^{b-1+\epsilon}}
			& \lesssim_{\delta,b} 
			\left(\sum_{k\in \mathbb{Z}^{\ast}}\int_{\mathbb{R}}
			|k|^{-2\left(\frac{1}{2}-\epsilon\right)\delta} \left\langle\frac{\tau-\lambda_{k}} {\langle k\rangle^{\delta}} \right\rangle^{2(b-1+\epsilon)} \left| \widehat{N_1u}(n,\tau)\right|^{2} d\tau\right)^{\frac{1}{2}}\\
			&=\left\|  |k|^{-\left(\frac{1}{2}-\epsilon\right)\delta} \left\langle\frac{\tau-\lambda_{k}} {\langle k\rangle^{\delta}} \right\rangle^{b-1+\epsilon}
			  \widehat{N_1u}(n,\tau)  \right\|_{L^{2}_{\tau}l^{2}_{k}},
	\end{split}}
	$$
	where, here and throughout, we use $L^2_\tau l^2_k$ to indicate $L^2_\tau l^2_k(\mathbb{R}\times\mathbb{Z}^*)$.
By	setting 
	\begin{equation*}
		f(n,\tau):=\langle n \rangle^{\frac{\delta}{2}} \left\langle\frac{\tau-\lambda_{n}} {\langle n\rangle^{\delta}} \right\rangle^{b}|\widehat{u}(n,\tau)|,
	\end{equation*}
	and 
\begin{equation}\label{stiN4}
\mathcal{M}=	\mathcal{M}_{l,n,k}(\tau):= |k|^{-\left(\frac{1}{2}-\epsilon\right)\delta}
\left\langle\frac{\tau-\lambda_{k}} {\langle k\rangle^{\delta}} \right\rangle^{b-1+\epsilon}
|l|^{\delta}
\langle n \rangle^{-\frac{\delta}{2}} \left\langle\frac{\tau-\lambda_{n}} {\langle n\rangle^{\delta}} \right\rangle^{-b},
\end{equation}
we obtain, using the definition of $N_{1}$,
\begin{equation*}
{
\begin{split}
\|D^{-\delta \left(b-\frac{1}{2}\right) }  N_{1}(u)\|_{Z^{b-1+\epsilon}}
& \lesssim_{\delta,b} 
\left\| 
\sum_{l\in \mathbb{Z}} \sum_{\substack{n\in \mathbb{Z}\\	n\neq k}} |\widehat{g}(k-l)||\widehat{g}(l-n)| \mathcal{M}_{l,n,k}(\tau)  f(n,\tau)  \right\|_{L^{2}_{\tau}l^{2}_{k}}\\
&\leq 
\left\| 
\sum_{\substack{l\in \mathbb{Z}\\ l\nsim k}} \sum_{\substack{n\in \mathbb{Z}\\	n\neq k}} |\widehat{g}(k-l)||\widehat{g}(l-n)| \mathcal{M}_{l,n,k}(\tau)  f(n,\tau)  \right\|_{L^{2}_{\tau}l^{2}_{k}}\\	
&\quad + \left\| 
\sum_{\substack{l\in \mathbb{Z}\\ l\nsim n}} \sum_{\substack{n\in \mathbb{Z}\\	n\neq k}} |\widehat{g}(k-l)||\widehat{g}(l-n)| \mathcal{M}_{l,n,k}(\tau)  f(n,\tau)  \right\|_{L^{2}_{\tau}l^{2}_{k}}\\
&\quad + \left\| 
\sum_{\substack{l\in \mathbb{Z}\\ k\sim l \sim n}} \sum_{\substack{n\in \mathbb{Z}\\	n\neq k}} |\widehat{g}(k-l)||\widehat{g}(l-n)| \mathcal{M}_{l,n,k}(\tau)  f(n,\tau)  \right\|_{L^{2}_{\tau}l^{2}_{k}}\\
&=:I_{1}+I_{2}+I_{3}.
\end{split}}
\end{equation*}
The idea to estimate $I_1$ and $I_2$ is to choose $\epsilon$ small and to use the decay of the Fourier coefficients of $g$ to control the term 	$\mathcal{M}_{l,n,k}(\tau)$. Indeed, recall that if $g^{(N)}$ denotes $N^{th}$ derivative of $g$ then $|\widehat{g}(k)|=|k|^{-N}|\widehat{g^{(N)}}(k)|$. Thus, using the estimate $|k-l|^{-N}\lesssim_{N} \max\{|k|,|l|\}^{-N}$, $N\in \mathbb{N}$, we infer 
\begin{equation*}
{
\begin{split}
I_{1}&\lesssim_{N}	
\left\| 
\sum_{\substack{l\in \mathbb{Z}\\ l\nsim k}} \sum_{\substack{n\in \mathbb{Z}\\	n\neq k}} |\widehat{g^{N}}(k-l)|\max\{|k|,|l|\}^{-N}|\widehat{g}(l-n)| \mathcal{M}_{l,n,k}(\tau)  f(n,\tau)  \right\|_{L^{2}_{\tau}l^{2}_{k}}.	
\end{split}}
\end{equation*}
Observe that if we request $\epsilon$ small enough such that $b-1+\epsilon<0$ then $\mathcal{M}_{l,n,k}(\tau)\leq |l|^{\delta}.$ Hence, taking $N>\delta$, we obtain
$$
\max\{|k|,|l|\}^{-N}\mathcal{M}_{l,n,k}(\tau)\leq \max\{|k|,|l|\}^{-N}|l|^{\delta}\leq |l|^{\delta-N}\leq 1,
$$
which implies
\begin{equation}\label{stiN8}
{
\begin{split}
I_{1}&\leq	
\left\| 
\sum_{\substack{l\in \mathbb{Z}\\ l\nsim k}} \sum_{\substack{n\in \mathbb{Z}\\	n\neq k}} |\widehat{g^{N}}(k-l)||\widehat{g}(l-n)|  f(n,\tau)  \right\|_{L^{2}_{\tau}l^{2}_{k}}\\
	&\leq	
\left\| \left[ |\widehat{g^{N}}(\cdot)|\ast \left( |\widehat{g}(\cdot)| \ast  f(\cdot,\tau) \right) \right](k)\right\|_{L^{2}_{\tau}l^{2}_{k}}\\
&\leq 
\left\|\widehat{g^{N}}(k)
\right\|_{l^{1}_{k}(\mathbb{Z})}
\left\|\widehat{g}(k) \right\|_{l^{1}_{k}(\mathbb{Z})}
\left\|f(k,\tau) \right\|_{L^{2}_{\tau}l^{2}_{k}}\\
&\lesssim_{g}\|u\|_{Z^{b}}\\
&\lesssim \|v\|_{Z^{b}_{T}},
\end{split}}
\end{equation}
where we have used Young's inequality and the fact that $\dfrac{1}{2}<b$ implies $\|f\|_{L^{2}_{\tau}l^{2}_{k}}\leq \|u\|_{Z^b}$.
Similarly, we can use the decay offered by the term $|\widehat{g}(l-n)|$ when $l\nsim n,$ to prove that 
\begin{equation*}
	I_{2}\lesssim_{g} \|v\|_{Z^{b}_{T}}.
\end{equation*}
	
It remains to deal with the estimate of $I_{3}.$ In this term we have that $k \sim l \sim n.$ Thus
\begin{equation}\label{stiN11}
|k|^{-\left(\frac{1}{2}-\epsilon\right)\delta}
|l|^{\delta}
\langle n \rangle^{-\frac{\delta}{2}} \sim 
\langle k \rangle^{\epsilon \delta}, 
\end{equation}
and it is necessary to recover $\epsilon \delta$ derivatives from the remaining terms in $\mathcal{M}$ (see \eqref{stiN4}). For this, we split the summation in $k$ into  high and low frequencies. Let $a$ be a positive constant such that Lemma \ref{stiN} holds for $|k|>a$. Then we may write
\begin{equation*}
{
\begin{split}
I_{3}
& \lesssim  \left( \sum_{\substack{k\in \mathbb{Z}^{\ast}\\ |k|\leq a}} \int_{\mathbb{R}}  \left[ 
\sum_{\substack{l\in \mathbb{Z}\\ k\sim l \sim n}} \sum_{\substack{n\in \mathbb{Z}\\	n\neq k}} |\widehat{g}(k-l)||\widehat{g}(l-n)| \mathcal{M}_{l,n,k}(\tau)  f(n,\tau)  \right]^{2}
d\tau \right)^{\frac{1}{2}}\\
&\quad + \left( \sum_{\substack{k\in \mathbb{Z}^{\ast}\\ |k|> a}} \int_{\mathbb{R}}  \left[ 
\sum_{\substack{l\in \mathbb{Z}\\ k\sim l \sim n}} \sum_{\substack{n\in \mathbb{Z}\\	n\neq k}} |\widehat{g}(k-l)||\widehat{g}(l-n)| \mathcal{M}_{l,n,k}(\tau)  f(n,\tau)  \right]^{2}
d\tau \right)^{\frac{1}{2}}\\
&=:I_{4}+I_{5}.
\end{split}}
\end{equation*}
If $|k|\leq a$ then from \eqref{stiN11} and $b-1+\epsilon<0$, we deduce that $\mathcal{M}$ is bounded by a constant depending on $a$. Therefore, a similar estimate as in \eqref{stiN8} yields
\begin{equation*}
{
\begin{split}
I_{4}
	& \lesssim_{a,\epsilon,\delta}
\left( \sum_{\substack{k\in \mathbb{Z}^{\ast}\\ |k|\leq a}} \int_{\mathbb{R}}  \left[ 
\sum_{\substack{l\in \mathbb{Z}\\ k\sim l \sim n}} \sum_{\substack{n\in \mathbb{Z}\\	n\neq k}} |\widehat{g}(k-l)||\widehat{g}(l-n)| 
f(n,\tau)  \right]^{2}
d\tau \right)^{\frac{1}{2}}\lesssim_{g}\|v\|_{Z^{b}_{T}}.
\end{split}}
\end{equation*}
	
To estimate $I_{5}$ we use Lemma \ref{stiN} and the fact that $\displaystyle{-b<-\frac{1}{2}<b-1<b-1+\epsilon<0}$ to obtain
\begin{equation*}
{
\begin{split}
\left\langle\frac{\tau-\lambda_{k}} {\langle k\rangle^{\delta}} \right\rangle^{b-1+\epsilon}
\left\langle\frac{\tau-\lambda_{n}} {\langle n\rangle^{\delta}} \right\rangle^{-b}
	&\leq 
\left\langle\frac{\tau-\lambda_{k}} {\langle k\rangle^{\delta}} \right\rangle^{b-1+\epsilon}
\left\langle\frac{\tau-\lambda_{n}} {\langle n\rangle^{\delta}} \right\rangle^{b-1+\epsilon}\\
	&\leq 
\max\left\{ \left\langle \frac{\tau-\lambda_{k}}{\langle k \rangle^{\delta} } \right\rangle, \left\langle \frac{\tau-\lambda_{n}}{\langle n \rangle^{\delta} } \right\rangle \right \}^{b-1+\epsilon}\\
	&\lesssim_{\beta,\delta,m}  \max\left\{ \langle k \rangle, \langle n \rangle \right\}^{(2m-\delta)(b-1+\epsilon )}\\
	&\sim \langle k \rangle^{(2m-\delta)(b-1+\epsilon )}.
\end{split}}
\end{equation*}
Thus, in this case
$
\mathcal{M}\lesssim 	\langle k\rangle^{\epsilon \delta} \langle k \rangle^{(2m-\delta)(b-1+\epsilon )}\lesssim 1,
$
provided $\epsilon$ is sufficiently small. From this we obtain
\begin{equation*}
{
\begin{split}
I_{5}
\leq
\left( \sum_{\substack{k\in \mathbb{Z}^{\ast}\\ |k|> a}} \int_{\mathbb{R}}  \left[ 
\sum_{\substack{l\in \mathbb{Z}\\ k\sim l \sim n}} \sum_{\substack{n\in \mathbb{Z}\\	n\neq k}} |\widehat{g}(k-l)||\widehat{g}(l-n)| 
f(n,\tau)  \right]^{2}
d\tau \right)^{\frac{1}{2}}
\lesssim_{g}\|v\|_{Z^{b}_{T}},
\end{split}}
\end{equation*}
and the proof of the theorem is completed.
\end{proof}

\begin{prop}\label{stiR1}
Let $\delta, b\in \mathbb{R}$ be given with $\delta>0$ and $b\in \left(\frac{1}{2}, 1\right).$ Let $g$ as defined in \eqref{OPG1}. For  $0<T<1$ assume $v\in Z_{T}^{b}$. Then there exists $0<\epsilon<\frac{1}{2}$ such that
	$$\displaystyle{\|D^{-\delta \left(b-\frac{1}{2}\right) }  R(v)\|_{Z_{T}^{b-1}} \lesssim_{\epsilon,b,\delta, g} T^{\frac{3}{2}-b-\epsilon} \|v\|_{Z^{b}_{T}}.}$$
\end{prop}
\begin{proof}
	This is consequence of Proposition \ref{locdamp7}, Lemma \ref{locdamp5}, and the embedding $Z^{b} \hookrightarrow Z^{\frac{1}{2}-\epsilon}$. Since the proof is similar to that of Lemma 2 in   \cite{Flores OH and Smith}, we omit the details.
\end{proof}

\subsection{Nonlinear estimates}
We start this subsection deducing a key bilinear estimate which is fundamental to estimate the nonlinear term $\partial_{x}(v^{2})$  in the dissipation-normalized Bourgain spaces. As in the proof of Theorem \ref{stiN1}, at some point we need to estimate the term 
\begin{equation}\label{bilin}
	\|D^{-\delta \left(b-\frac{1}{2}\right) }  \partial_{x}(uv)\|_{Z^{b-1+\epsilon}}
\end{equation}
for some $\epsilon>0$. But, from Proposition \ref{prop2}, we have $(Z^{b-1+\epsilon})^{\ast}=Z^{1-b-\epsilon}.$ Thus, we can use duality and Plancherel's theorem to rewrite \eqref{bilin} in the following form
\begin{equation}\label{bilin1}
	{
		\begin{split}
			\|D^{-\delta \left(b-\frac{1}{2}\right) }  \partial_{x}(uv)\|_{Z^{b-1+\epsilon}}
		=
			\sup _{w\in S_{\epsilon}^b} \left|  
			\int_{\Gamma} k_{3}|k_{3}|^{-\delta \left(b-\frac{1}{2}\right)}
			\widehat{u}(k_{2},\tau_{2}) \widehat{v}(k_{1},\tau_{1})\;
			\widehat{w}(k_{3},\tau_{3}) \; dS \right|,\\
	\end{split}}
\end{equation}
where
	$$\Gamma:= \{(\tau_{1},\tau_{2},\tau_{3},k_{1},k_{2},k_{3}) :
\tau_{j}\in \mathbb{R}, k_{j}\in \mathbb{Z},\; \tau_{1}+\tau_{2}+\tau_{3}=0,\; k_{1}+k_{2}+k_{3}=0, k_{1}k_{2}k_{3}\neq 0 \}$$
 and  $dS$ is the inherited measure on the plane $\tau_{1}+\tau_{2}+\tau_{3}=0$. In addition, $S_{\epsilon}^b$ is the unit sphere in $Z^{1-b-\epsilon}$, that is,
 $$
 S_{\epsilon}^b=\{w\in Z^{1-b-\epsilon}; \|w\|_{Z^{1-b-\epsilon}}=1\}.
 $$

In what follows, to simplify the exposition, we adopt the following notation introduced by Tao in \cite{Tao}: for $j=1,2,3$, let $N_{j}>0$ be dyadic numbers such that  $|k_{j}|\sim N_{j}$. So, $N_{j}$ will measure the magnitude of frequencies of the waves. It is convenient to introduce the quantities $N_{\max}\geq N_{\med}\geq N_{\min}$ to be the maximum, median, and minimum of $N_{1},N_{2},N_{3}$, respectively.  Similarly,  let  $L_{j}>0$ be dyadic numbers such that  $\displaystyle{\left|\frac{\tau_{j}-\lambda_{k_{j}}} {\langle k_{j}\rangle^{\delta}} \right| \sim L_{j}  }$ so that $L_j$ is the  $j^{th}$ ``modulational'' frequency. We also set $L_{\max}\geq L_{\med}\geq L_{\min}$ to denote the maximum, median, and minimum of $L_{1},L_{2},L_{3}$, respectively.

The following two lemmas will be useful to prove our main estimate.

\begin{lem}\label{bilin3}
Let $ k_{1},k_{2},k_{3} \in \mathbb{Z}$ be given with $k_{1}+k_{2}+k_{3}=0$ and $k_{1}k_{2}k_{3}\neq 0.$  If $\lambda_{k_{j}}= -\beta k_{j}|k_{j}|^{2m}+\alpha k_{j} |k_{j}|^{2r}-2 \mu k_{j}$ then
	$$\left|\sum_{j=1}^{3}\lambda_{k_{j}}\right|
	\gtrsim_{\beta,m} N_{\max}^{2m} \cdot N_{\min},$$
	provided 
		$N_{\max}$ is sufficiently large.
\end{lem}
\begin{proof}
	Without loss of generality, we assume that $|k_{1}|\geq |k_{2}| \geq |k_{3}|.$ In view of the identity $k_{1}+k_{2}+k_{3}=0$ we infer that both $k_{2}$ and $k_{3}$ share the same sign which is opposite to that of $k_{1}.$ Moreover, $|k_1|\sim|k_2|$ and the identity $|k_{3}|=|k_{1}|-|k_{2}|$ holds.
	
Let us assume that $k_1<0$ (the case $k_1>0$ may be handled similarly). Hence we must have $k_2>0$ and $k_3>0$.
We now split into two cases.
	
\noindent	\textbf{Case 1.} $|k_{1}| \sim |k_{2}|\gg |k_{3}|.$ 
	
From the Mean Value Theorem	there exists 
$k^{\ast} \in (|k_{2}|,|k_{1}|)$ such that
\begin{equation*}
{
\begin{split}
\left|\sum_{j=1}^{3}\lambda_{k_{j}}\right|
	&= 
\left| \beta |k_{1}|^{2m+1}-\alpha |k_{1}|^{2r+1} -(\beta|k_{2}|^{2m+1}-\alpha |k_{2}|^{2r+1}) -\beta |k_{3}|^{2m+1}+\alpha |k_{3}|^{2r+1}\right|\\
	&= 
\left| \left(\beta (2m+1) (k^{\ast})^{2m}-\alpha (2r+1) (k^{\ast})^{2r} \right) ( |k_{1}|-|k_{2}|) -\beta |k_{3}|^{2m+1}+\alpha |k_{3}|^{2r+1}\right|\\
	&= 
\left| \beta (2m+1) (k^{\ast})^{2m} \left(1-\frac{\alpha (2r+1) (k^{\ast})^{2(r-m)}}{\beta (2m+1)} \right)  |k_{3}| -\beta |k_{3}|^{2m+1}+\alpha |k_{3}|^{2r+1}\right|.\\
\end{split}}
\end{equation*}	
Note that for $|k_1|$ (and hence $|k_2|$) large, we have $(k^*)^{2(r-m)}$ small. Thus, using that $k^*\sim k_1$ we obtain
\begin{equation*}
{
\begin{split}
\left|\sum_{j=1}^{3}\lambda_{k_{j}}\right|
	&\gtrsim
|k_1|^{2m}  |k_{3}| -\beta |k_{3}|^{2m+1}+\alpha |k_{3}|^{2r+1}\\
	&\geq 
|k_{1}|^{2m} \left(1  -\beta \left(\frac{|k_{3}|}{|k_{1}|}\right)^{2m} \right)|k_{3}| \\
	&\gtrsim_{\beta,m} N_{\max}^{2m} \cdot N_{\min},
\end{split}}
\end{equation*}	
where we used that $|k_1|\gg|k_3|$ is the last inequality.
	
\noindent	\textbf{Case 2.} $|k_{1}| \sim |k_{2}| \sim |k_{3}|.$ 
	
Using that $k_2=-k_1-k_3$, from the Mean Value Theorem	there exist
$k^{\ast} \in (|k_{2}|,|k_{1}|)$ and
$k^{\ast \ast} \in (|k_{3}|,|k_{2}|)$ such that
\begin{equation}\label{bilin12}
{
\begin{split}
\left|\sum_{j=1}^{3}\lambda_{k_{j}}\right|
	&= 
\bigg|-k_{1} \Big[\beta |k_{1}|^{2m}-\alpha |k_{1}|^{2r} -(\beta|k_{2}|^{2m}-\alpha |k_{2}|^{2r}) \Big] \\
	&\quad+
k_{3} \Big[\beta |k_{2}|^{2m}-\alpha |k_{2}|^{2r} -(\beta|k_{3}|^{2m}-\alpha |k_{3}|^{2r}) \Big] \bigg|\\
	&= 
\bigg|	|k_{1}| \Big(2m\beta (k^{\ast})^{2m-1}-2r\alpha (k^{\ast})^{2r-1} \Big) ( |k_{1}|-|k_{2}|)\\  
&\quad+ 
|k_{3}| \Big(2m\beta  (k^{\ast \ast})^{2m-1}-2r\alpha (k^{\ast \ast})^{2r-1} \Big) ( |k_{2}|-|k_{3}|) \bigg|
				\\
	&= 
\bigg|2m\beta	|k_{1}| (k^{\ast})^{2m-1} \left(1-\frac{2r\alpha  (k^{\ast})^{2(r-m)}}{2m\beta} \right)  |k_{3}|\\ 
	&\quad+	2m\beta|k_{3}| (k^{\ast \ast})^{2m-1} \left(1-\frac{2r\alpha (k^{\ast \ast})^{2(r-m)}}{2m\beta} \right)  (|k_{2}|-|k_{3}|)\bigg|.
				\\
\end{split}}
\end{equation}
Since $k^*$ and $k^{**}$ are large, both terms on the right-hand side of \eqref{bilin12} are positive. So, we can discard the second one and use that $k^*\sim k_1$ to obtain
\begin{equation*}
{
\begin{split}
\left|\sum_{j=1}^{3}\lambda_{k_{j}}\right|
	\gtrsim
|k_{1}| |k_{1}|^{2m-1}   |k_{3}|
\gtrsim_{\beta,m} N_{\max}^{2m} \cdot N_{\min}. 
\end{split}}
\end{equation*}
	This completes the proof of the lemma.
\end{proof}

\begin{lem}\label{bilin19}
Assume $\delta>0$.	Let  $ \tau_{1},\tau_{2},\tau_{3} \in \mathbb{R}$ and $ k_{1},k_{2},k_{3} \in \mathbb{Z}$ be given with  $\tau_{1}+\tau_{2}+\tau_{3}=0,$ $k_{1}+k_{2}+k_{3}=0,$ and $k_{1}k_{2}k_{3}\neq 0.$  Then
	\begin{equation}\label{bilin20}
		L_{\max}
		\gtrsim_{\beta,m, \delta} N_{\max}^{2m-\delta} \cdot N_{\min},
	\end{equation}	 
	provided $N_{\max}$ is sufficiently large.

	Furthermore, if $L_{\max}$ occurs at the same index as $N_{\min},$ i.e., $N_{j_{0}}=N_{\min}$ and $L_{j_{0}}=L_{\max}$ for some $j_{0}\in \{1,2,3\},$ then we have 
		\begin{equation}\label{bilin26}
			L_{\max}
			\gtrsim_{\beta,m, \delta} N_{\max}^{2m} \cdot N_{\min}^{1-\delta},
		\end{equation}
	whether
		 \begin{equation}\label{bilin22.1}
		 	\max\left\{L_{\min},L_{\med},L_{\max} N_{\min}^{\delta} N_{\max}^{-\delta} \right\}= L_{\max} N_{\min}^{\delta} N_{\max}^{-\delta},
		 \end{equation}
and
		\begin{equation}\label{bilin52}
			L_{\max} \geq L_{\med}
			\gtrsim_{\beta,m, \delta} N_{\max}^{2m-\delta} \cdot N_{\min},
		\end{equation}
	whether \begin{equation}\label{bilin22.2}
		\max\left\{L_{\min},L_{\med},L_{\max} N_{\min}^{\delta} N_{\max}^{-\delta} \right\}\neq L_{\max} N_{\min}^{\delta} N_{\max}^{-\delta}.  
	\end{equation}
		
\end{lem}
\begin{proof}
	Without loss of generality  we also assume $|k_{1}|\geq |k_{2}| \geq |k_{3}|.$ Thus, $N_{\max}= N_{1}$ and $N_{\min}= N_{3}.$  As we have seen in the proof of Lemma \ref{bilin3}, in any case we always have $|k_{1}|\sim |k_{2}|.$ Hence, $N_{1}\sim N_{2}$ and
	$\displaystyle{ \frac{N_{3}}{N_{1}}=\frac{N_{\min}}{N_{\max}}\leq 1}.$
Thus, if  $\lambda_{k_{j}}= -\beta k_{j}|k_{j}|^{2m}+\alpha k_{j} |k_{j}|^{2r}-2 \mu k_{j},$ using that $\tau_{1}+\tau_{2}+\tau_{3}=0$ we may write
\begin{equation}\label{bilin22}
{
\begin{split}
\Bigg|\frac{\sum\limits_{j=1}^{3}\lambda_{k_{j}}}{\langle k_{1}\rangle^{\delta}}\Bigg|
	&=
\left|\frac{\tau_{1}+\lambda_{k_{1}}}{\langle k_{1}\rangle^{\delta}  } + \frac{\tau_{2}+\lambda_{k_{2}}}{\langle k_{1}\rangle^{\delta}  } + \frac{\tau_{3}+\lambda_{k_{3}}}{\langle k_{1}\rangle^{\delta}  } \right|\\
	& \lesssim 
\left|\frac{\tau_{1}+\lambda_{k_{1}}}{\langle k_{1}\rangle^{\delta}  }\right|
+ \left|\frac{\tau_{2}+\lambda_{k_{2}}}{\langle k_{2}\rangle^{\delta}  }\right|
+\left|\frac{\tau_{3}+\lambda_{k_{3}}}{\langle k_{3}\rangle^{\delta}  }\right| 
\frac{ \langle k_{3}\rangle^{\delta}} {\langle k_{1}\rangle^{\delta}}\\
	& \lesssim
L_{1}+L_{2}+L_{3} N_{\min}^{\delta} N_{\max}^{-\delta}\\
& \leq \max\left\{L_{1},L_{2},L_{3}N_{\min}^{\delta} N_{\max}^{-\delta}   \right\}\\
&  \lesssim L_{\max}.
\end{split}}	
\end{equation}

On the other hand, from Lemma \ref{bilin3} we infer
\begin{equation}\label{bilin23}
{
\begin{split}
\Bigg|\frac{\sum\limits_{j=1}^{3}\lambda_{k_{j}}} {\langle k_{1}\rangle^{\delta}}\Bigg|
\gtrsim_{\beta,m, \delta} \frac{ N_{\max}^{2m} \cdot N_{\min}} {N_{\max}^{\delta}}
= N_{\max}^{2m-\delta} \cdot N_{\min}.
\end{split}}
\end{equation}
	From  \eqref{bilin22} and \eqref{bilin23} we get
	\eqref{bilin20}.
	
	Now, note that if $L_{\max}=L_{3}$ and \eqref{bilin22.1} holds
	then
	\eqref{bilin22} and \eqref{bilin23} imply
	$${
		\begin{split}
			\frac{N_{\max}^{2m} \cdot N_{\min}}{N_{\max}^{\delta} }
			&\lesssim 
			L_{\max} N_{\min}^{\delta} N_{\max}^{-\delta}.
	\end{split}}$$
which gives \eqref{bilin26}. On the other hand, if \eqref{bilin22.1} does not occur then the maximum must be $L_{\med}$ and \eqref{bilin52} holds.
	This completes the proof of the lemma.
\end{proof}

Next result is our main bilinear estimate.

\begin{thm}\label{bilinear}
	Assume  $0<\delta \leq 1$  satisfies
	\begin{equation*}
		2m>2-\delta.
	\end{equation*}	
	For $0<T<1$ and $b$ satisfying
	\begin{equation}\label{minb}
		\frac{1}{2}<b<\min\left\{1,  \frac{2m-\frac{1}{2}}{2m-\delta+1} \right\},
	\end{equation}
	assume that  $\tilde{u},\tilde{v}\in Z^{b}_{T}$. Then there exists $\epsilon>0$ small  such that
	\begin{equation*}
		\|D^{-\delta \left(b-\frac{1}{2}\right) }  \partial_{x}(\tilde{u}\tilde{v})\|_{Z_{T}^{b-1}}
		\lesssim_{\epsilon,b} T^{\epsilon} \|\tilde{u}\|_{Z_{T}^{b}} \|\tilde{v}\|_{Z_{T}^{b}}.	
	\end{equation*}
\end{thm}
\begin{proof}
	Let $\epsilon>0$ to be chosen later.	 Let $u,v:\mathbb{T}\times \mathbb{R}\to \mathbb{R}$ be functions in $Z^{b}$ such that $\tilde{u}(t)\equiv u(t),$ (resp. $\tilde{v}(t)\equiv v(t)$) on $[0,T]$ with 
	$\|\tilde{u}\|_{Z^{b}}\leq 2 \|u\|_{Z_{T}^{b}}$ 
	(resp. $\|\tilde{v}\|_{Z^{b}}\leq 2 \|v\|_{Z_{T}^{b}}$). 
	
	For $\epsilon<\dfrac{1}{2}$ and $\dfrac{1}{2}<b<1$ we have $\displaystyle{-\frac{1}{2}<b-1<b-1+\epsilon <\frac{1}{2}}$. Thus, Proposition \ref{locdamp7} yields
	\begin{equation*}
		{\normalsize
			\begin{split}
				\|D^{-\delta \left(b-\frac{1}{2}\right) }  \partial_{x}(\tilde{u}\tilde{v})\|_{Z_{T}^{b-1}}	
				&\lesssim_{\eta,b,\epsilon}T^{\epsilon} \|D^{-\delta \left(b-\frac{1}{2}\right) }  \partial_{x}(uv)\|_{Z^{b-1+\epsilon}}.
		\end{split}}
	\end{equation*}
	So we need to estimate the right-hand side of the last inequality. Set $I:=\|D^{-\delta \left(b-\frac{1}{2}\right) }  \partial_{x}(uv)\|_{Z^{b-1+\epsilon}} $ and define
\begin{equation*}
\widehat{f}(k_{2},\tau_{2}):=|k_{2}|^{\frac{\delta}{2}} \left\langle\frac{\tau_{2}-\lambda_{k_{2}}} {\langle k_{2}\rangle^{\delta}} \right\rangle^{b}
|\widehat{u}(k_{2},\tau_{2})|,
\end{equation*}
\begin{equation*}
\widehat{g}(k_{1},\tau_{1}):=|k_{1}|^{\frac{\delta}{2}} \left\langle\frac{\tau_{1}-\lambda_{k_{1}}} {\langle k_{1}\rangle^{\delta}} \right\rangle^{b} 
|\widehat{v}(k_{1},\tau_{1})|,
\end{equation*}
and
\begin{equation*}
\widehat{h}(k_{3},\tau_{3}):=|k_{3}|^{1-b-\epsilon} \left\langle\frac{\tau_{3}-\lambda_{k_{3}}} {\langle k_{3}\rangle^{\delta}} \right\rangle^{1-b-\epsilon} 
|\widehat{w}(k_{3},\tau_{3})|.
\end{equation*}
From \eqref{bilin1} we get
\begin{equation}\label{bilinear4}
{
\begin{split}
I
&\leq
\sup_{w\in S_\epsilon^b }  \left( 
\int_{\Gamma}
\frac{  \widehat{f}(k_{2},\tau_{2})  \widehat{g}(k_{1},\tau_{1})
\widehat{h}(k_{3},\tau_{3})
|k_{3}|^{1-\delta \left(b-\frac{1}{2}\right)-1+b+\epsilon}}
{
|k_{1}|^{\frac{\delta}{2}} 
|k_{2}|^{\frac{\delta}{2}} 
\left\langle\frac{\tau_{1}-\lambda_{k_{1}}} {\langle k_{1}\rangle^{\delta}} \right\rangle^{b}
\left\langle\frac{\tau_{2}-\lambda_{k_{2}}} {\langle k_{2}\rangle^{\delta}} \right\rangle^{b}
\left\langle\frac{\tau_{3}-\lambda_{k_{3}}} {\langle k_{3}\rangle^{\delta}} \right\rangle^{1-b-\epsilon}
}
dS \right).
\end{split}}
\end{equation}
	
In order to estimate the above term we need localization in both frequencies and modulations. Therefore, we define a partition of unity as follows:  fix a smooth radial function $\varphi\in C_{0}^{\infty}(\mathbb{R})$ such that $	0\leq \varphi \leq 1$, $\varphi(\xi)=1$ if $|\xi|\leq1$, and $\text{supp}(\varphi) \subset \left\{\xi \in \mathbb{R}: \;  |\xi| \leq 2\right\}$.
For $i\in \mathbb{N}^{\ast},$ we define $$\phi(\xi)=\varphi(\xi)-\varphi(2\xi),\;\; \displaystyle{\phi_{2^{i}}(\xi)=\phi(2^{-i}\xi)},\;\; \text{ and}\;\; \displaystyle{\rho_{2^{i}}(\tau,k):=\phi\left(2^{-i} \frac{\tau-\lambda_{k}} {\langle k\rangle^{\delta}}  \right)}. $$ 
 Observe that $\text{supp}(\phi_{2^{i}})\subset \left\{ \xi\in \mathbb{R}: 2^{i-1}<|\xi|<2^{i+1} \right\}.$ By convention, we denote 
	$$\displaystyle{\phi_{1}(\xi)=\varphi(\xi)}\;\; \text{ and}\;\;
	\displaystyle{\rho_{1}(\tau,k):=\varphi\left( \frac{\tau-\lambda_{k}} {\langle k\rangle^{\delta}}  \right)}. $$
In what follows, any summation over capitalized variables such as $N$ and $L$ are presumed to be dyadic with $N,L\geq 1$, i.e., these variables range over the numbers of the form $\{2^{i}:i\in \mathbb{N}\}.$ Then, we have
$$\sum_{N\geq 1}\phi_{N}(\xi)=1 \quad \mbox{and}\quad \;\;\text{supp}(\phi_{N})\subset \left\{ \xi\in \mathbb{R}: \frac{N}{2}<|\xi|<2N \right\},\;\;N\geq 2. $$
Next, we  define the Littlewood-Paley multipliers by 
\begin{equation}\label{bilinear15}
	\widehat{P_{N}u}(k)=\phi_{N}(k)\widehat{u}(k),
\end{equation}
\begin{equation}\label{bilinear16}
	\widehat{Q_{L}u}(k,\tau)=\rho_{L}(k,\tau)\widehat{u}(k,\tau).
\end{equation}
It must be clear that in \eqref{bilinear15} the Fourier transform is taken only over the spatial variable while in \eqref{bilinear16} it is taken over both spatial and time variables.
	
Using these multipliers, we localize in frequencies and modulations and rewrite \eqref{bilinear4} as
\begin{equation*}
{
\begin{split}
I
	&\lesssim
\sup_{w\in S_\epsilon^b } \left(  \sum_{\substack{N_1,N_2,N_3\\ 
L_1,L_2,L_3 }} 
\int\limits_{\Gamma }
I_{N_1,N_2,N_3}^{L_1,L_2,L_3} \widehat{P_{N_{2}}Q_{L_{2}}f}(k_{2},\tau_{2}) \widehat{ P_{N_{1}}Q_{L_{1}} g}(k_{1},\tau_{1})
\widehat{ P_{N_{3}}Q_{L_{3}} h}(k_{3},\tau_{3})
dS \right),
\end{split}}
\end{equation*}
where 
$$
I_{N_1,N_2,N_3}^{L_1,L_2,L_3}:=\frac{N_{3}^{1-\delta \left(b-\frac{1}{2}\right)-1+b+\epsilon}}{	N_{1}^{\frac{\delta}{2}} 
N_{2}^{\frac{\delta}{2}} 
L_{1}^{b}
L_{2}^{b}
L_{3}^{1-b-\epsilon}}.
$$

The following lemma is the main ingredient to obtain our estimate.
	\begin{lem}\label{bilinear27}
	Under the assumptions of Theorem \ref{bilinear} there exists
		$\epsilon>0$ small  such that
		\begin{equation}\label{bilinear19}
			{\begin{split}
					I_{N_1,N_2,N_3}^{L_1,L_2,L_3} \lesssim_{b,\epsilon,\delta} 
					\frac{1 } 
					{N_{\min}^{\frac{1}{2}+\epsilon}  N_{j_{1}}^{\frac{\delta}{2}}	L_{j_{1}}^{b'} L_{\max}^{\epsilon}},		
			\end{split}}
		\end{equation}
		for some $j_{1}\in \{1,2,3\}$ and some $b'$ with $\frac{1}{2}<b'\leq b.$
	\end{lem}
	\begin{proof}
	Let us start by supposing $\epsilon<1-b$.
Then, since $\delta\leq 1$ we deduce
$$	1-\delta \left(b-\frac{1}{2}\right)-1+b+\epsilon =1-\frac{\delta}{2}+\epsilon \delta + (\delta-1)(1-b-\epsilon)
\leq1-\frac{\delta}{2}+\epsilon \delta,$$ 
which implies
		\begin{equation}\label{bilinear9}
			N_{3}^{1-\delta \left(b-\frac{1}{2}\right)-1+b+\epsilon}
			\leq
			N_{3}^{1-\frac{\delta}{2}+\epsilon \delta}.
		\end{equation}
Next we claim that
	\begin{equation}\label{bilinear14}
		\begin{split}
			I_{N_1,N_2,N_3}^{L_1,L_2,L_3}
			&\lesssim 
			\frac{	N_{3}^{1-\frac{\delta}{2}+\epsilon \delta }} 
			{N_{1}^{\frac{\delta}{2}}  N_{2}^{\frac{\delta}{2}} 	L_{\med}^{b} L_{\min}^{b} L_{\max}^{1-b-\epsilon}}.	
	\end{split}
\end{equation}
Indeed, in view of \eqref{bilinear9} it suffices to show that
\begin{equation}\label{bilinear14.1}
	\frac{1}{L_{1}^{b}L_{2}^{b}L_{3}^{1-b-\epsilon}}\lesssim \frac{1}{L_{\min}^{b} L_{\med}^{b} L_{\max}^{1-b-\epsilon}}.
\end{equation}
	In order to see that \eqref{bilinear14.1} holds, we assume $L_{\max}=L_{1},$ $L_{\med}=L_{2},$ and 
		$L_{\min}=L_{3}$ (the other (five) cases are treated similarly). First we write
		$$
		{
			\begin{split}
				\frac{1}{L_{1}^{b}L_{2}^{b}L_{3}^{1-b-\epsilon}}
				& = \frac{1} {\frac{L_{1}^{b}}{L_{\min}^{b}} \frac{L_{2}^{b}}{L_{\med}^{b}} \frac{L_{3}^{1-b-\epsilon}}{L_{\max}^{1-b-\epsilon}}L_{\min}^{b} L_{\med}^{b} L_{\max}^{1-b-\epsilon}  }	\\
				&=\frac{1}{L_{\min}^{b} L_{\med}^{b} L_{\max}^{1-b-\epsilon}} \left(\frac{L_{1}}{L_{3}}\right)^{1-2b-\epsilon}.
		\end{split}}
		$$
	Since $\displaystyle{\frac{L_{1}}{L_{3}}  = \frac{L_{\max}}{L_{\min}} \gtrsim 1}$ and 	$1-2b-\epsilon<0,$ estimate
		\eqref{bilinear14.1} follows and the proof of the claim is completed.

	So, the general strategy is to estimate the right-hand side of \eqref{bilinear14}.
		First, we assume that $N_{\max}\geq a$, where $a$ is a sufficiently large constant such that Lemmas \ref{bilin3} and \ref{bilin19} hold. We analyse two cases:\\

	\noindent	{\bf Case 1:} $N_{\min}=N_{j_{0}}$ and $L_{\max}=L_{j_{0}^{\ast}}$ for some $j_{0},j_{0}^{\ast}\in \{1,2,3\}$ with $j_{0}\neq j_{0}^{\ast}.$ \\

If we request $	2\epsilon< 1-b$, from \eqref{bilinear14} and \eqref{bilin20} we obtain
		\begin{equation*}
			{
				\begin{split}
						I_{N_1,N_2,N_3}^{L_1,L_2,L_3}
					& \lesssim
					\frac{N_{3}^{1-\delta+\epsilon \delta} N_{3}^{\frac{\delta}{2}} } 
					{N_{1}^{\frac{\delta}{2}}  N_{2}^{\frac{\delta}{2}} 	L_{\med}^{b} L_{\min}^{b} L_{\max}^{1-b-2\epsilon}  
						L_{\max}^{\epsilon}  }\\
					& \lesssim 
					\frac{N_{\max}^{\frac{\delta}{2}} N_{\max}^{1-\delta+\epsilon \delta-(2m-\delta)(1-b-2\epsilon)} } 
					{N_{1}^{\frac{\delta}{2}} N_{2}^{\frac{\delta}{2}} N_{\min}^{1-b-2\epsilon}	L_{\med}^{b} L_{\min}^{b} L_{\max}^{\epsilon}}\\
					& \lesssim
					\frac{ N_{\max}^{1-\delta+\epsilon \delta-(2m-\delta)(1-b-2\epsilon)} } 
					{ N_{\min}^{\frac{\delta}{2}} N_{\min}^{\frac{1}{2}+\epsilon} N_{\min}^{\frac{1}{2}-b-3\epsilon}	L_{\med}^{b} L_{\min}^{b} L_{\max}^{\epsilon}}.		
			\end{split}}
		\end{equation*}
The idea now is to ensure that we can borrow enough remaining derivative from $N_{\max}$ to contribute to $N_{\min}$. Indeed, since ${\frac{1}{2}-b-3\epsilon<0},$ 
		\begin{equation}\label{bilinear29}
			{
				\begin{split}
						I_{N_1,N_2,N_3}^{L_1,L_2,L_3}
					&\lesssim	
					\frac{ N_{\max}^{1-\delta+\epsilon  \delta-(2m-\delta)(1-b-2\epsilon)}
						N_{\min}^{-\frac{1}{2}+b+3\epsilon} } 
					{ N_{\min}^{\frac{1}{2}+\epsilon}  N_{\min}^{\frac{\delta}{2}} 	L_{\min}^{b} L_{\med}^{b}  L_{\max}^{\epsilon}}\\
					& \leq
					\frac{ N_{\max}^{1-\delta+\epsilon  \delta-(2m-\delta)(1-b-2\epsilon)
							-\frac{1}{2}+b+3\epsilon} } 
					{ N_{\min}^{\frac{1}{2}+\epsilon}  N_{j_{1}}^{\frac{\delta}{2}} 	L_{j_{1}}^{b}  L_{\max}^{\epsilon}},		
			\end{split}}
		\end{equation}
		for some $j_{1}\in \{1,2,3\}.$ Note that, in this case, $j_{1}=j_{0}.$
		Now, to make the exponent of $N_{\max}$ in \eqref{bilinear29} negative, we impose
		\begin{equation}\label{bilinear36} 
			{\begin{split}
					1-\delta-(2m-\delta)(1-b)+b-\frac{1}{2}<0
					& \Longleftrightarrow b<\frac{2m-\frac{1}{2}}{2m-\delta+1}.
			\end{split}}
		\end{equation}
	This last condition on $b$ is compatible with the fact that $b>\frac{1}{2}$ because
		$$ 
		{\begin{split}
				\frac{2m-\frac{1}{2}}{2m-\delta+1}>\frac{1}{2}
				& \Longleftrightarrow 2m>2-\delta.					
		\end{split}}
		$$
	Hence, if we take $\epsilon$ sufficiently small we see that the power of $N_{\max}$ in \eqref{bilinear29} is negative and \eqref{bilinear19} holds.\\
		
	\noindent	{\bf Case 2:} If $N_{\min}=N_{j_{0}},$ $L_{\max}=L_{j_{0}}$ for some $j_{0}\in \{1,2,3\}.$
	
	Here we analyse two subcases according to Lemma \ref{bilin19}.\\

	\noindent	{\bf Subcase 2.1:} If \eqref{bilin22.1} holds.

Since $1-b-2\epsilon>0$,  \eqref{bilinear14}   and \eqref{bilin26} yield
		\begin{equation*}
			{\begin{split}
					I_{N_1,N_2,N_3}^{L_1,L_2,L_3}
					& \leq
					\frac{N_{3}^{1-\delta+\epsilon \delta} N_{3}^{\frac{\delta}{2}} } 
					{N_{1}^{\frac{\delta}{2}}  N_{2}^{\frac{\delta}{2}} 	L_{\med}^{b} L_{\min}^{b} L_{\max}^{1-b-2\epsilon} L_{\max}^{\epsilon} }\\
					& \lesssim 
					\frac{N_{\max}^{1-\delta+\epsilon \delta-(2m)(1-b-2\epsilon)} N_{\max}^{\frac{\delta}{2}} } 
					{ N_{1}^{\frac{\delta}{2}} N_{2}^{\frac{\delta}{2}} N_{\min}^{(1-\delta)(1-b-2\epsilon)}	L_{\med}^{b} L_{\min}^{b} L_{\max}^{\epsilon}}\\
					& \lesssim
					\frac{ N_{\max}^{1-\frac{\delta}{2}+\epsilon \delta-(2m)(1-b-2\epsilon)} } 
					{ N_{\min}^{\frac{\delta}{2}} N_{\max}^{\frac{\delta}{2}} N_{\min}^{(1-\delta)(1-b-2\epsilon)}	L_{\med}^{b} L_{\min}^{b} L_{\max}^{\epsilon}}.		
			\end{split}}
		\end{equation*}
	 Let $j_{1}\in \{1,2,3\}\setminus \{j_{0}\}.$ Since $L_{j_{1}}\neq L_{\max},$ either  $L_{j_{1}}=L_{med}$ or $L_{j_{1}}=  L_{\min}.$ Then we get
		\begin{equation*}
			{\begin{split}
					I_{N_1,N_2,N_3}^{L_1,L_2,L_3}
					&\lesssim
					\frac{ N_{\max}^{1-\frac{\delta}{2}+\epsilon \delta-(2m)(1-b-2\epsilon)} } 
					{  N_{j_{1}}^{\frac{\delta}{2}} L_{j_{1}}^{b} N_{\min}^{\frac{\delta}{2}+(1-\delta)(1-b-2\epsilon)}	  L_{\max}^{\epsilon}}\\
					&\lesssim 	\frac{ N_{\max}^{1-\frac{\delta}{2}+\epsilon \delta-2m(1-b-2\epsilon)-\frac{\delta}{2}-(1-\delta)(1-b-2\epsilon)+\frac{1}{2}+\epsilon} } 
					{  N_{j_{1}}^{\frac{\delta}{2}} L_{j_{1}}^{b} N_{\min}^{\frac{1}{2}+\epsilon}	  L_{\max}^{\epsilon}}.
			\end{split}}
		\end{equation*}
	In order to ensure that the exponent of $N_{\max}$ is negative,  we request 
		\begin{equation*}
		1-\frac{\delta}{2}-2m(1-b)-\frac{\delta}{2}-(1-\delta)(1-b)+\frac{1}{2}<0
		\end{equation*}
		which is equivalent to the condition on $b$ in \eqref{bilinear36}.  Hence, for $\epsilon$ small enough we see that \eqref{bilinear19} also holds in this case.\\

\noindent		{\bf Subcase 2.2:} If \eqref{bilin22.2} holds. 
		
		 First, we consider $j_{0}=3.$ Thus $N_{\min}=N_{3}$ and $L_{\max}=L_{3}.$ Using \eqref{bilinear14},
		\begin{equation*}
			{\begin{split}
					I_{N_1,N_2,N_3}^{L_1,L_2,L_3}
					& \leq
					\frac{N_{3}^{1-\frac{\delta}{2}+\epsilon \delta} } 
					{N_{\max}^{\frac{\delta}{2}}  N_{j_{1}}^{\frac{\delta}{2}} 	L_{j_{1}}^{b}  L_{\max}^{1-b-\epsilon}  }\\
					& =
					\frac{N_{\min}^{\frac{3}{2}-\frac{\delta}{2}+\epsilon \delta+\epsilon} } 
					{N_{\min}^{\frac{1}{2}+\epsilon}N_{\max}^{\frac{\delta}{2}}  N_{j_{1}}^{\frac{\delta}{2}} 	L_{j_{1}}^{b}  L_{\max}^{1-b-2\epsilon} L_{\max}^{\epsilon} },
			\end{split}}
		\end{equation*}
		where $j_{1}\in \{1,2\}.$ Since $1-b-2\epsilon<0$, from \eqref{bilin52}, we obtain
		\begin{equation*}
			{\begin{split}
					I_{N_1,N_2,N_3}^{L_1,L_2,L_3}
					& \lesssim
					\frac{N_{\min}^{\frac{1}{2}-\frac{\delta}{2}+b+\epsilon( \delta+3)} } 
					{N_{\min}^{\frac{1}{2}+\epsilon}   N_{j_{1}}^{\frac{\delta}{2}}
						L_{j_{1}}^{b} 
						N_{\max}^{\frac{\delta}{2} +(2m-\delta)(1-b-2\epsilon)}
						L_{\max}^{\epsilon} }	\\
					& \leq
					\frac{N_{\min}^{\frac{1}{2}-\delta+b -(2m-\delta)(1-b)
							+\epsilon(4m-\delta+3)} } 
					{N_{\min}^{\frac{1}{2}+\epsilon}   N_{j_{1}}^{\frac{\delta}{2}}
						L_{j_{1}}^{b} 
						L_{\max}^{\epsilon} }.	    
			\end{split}}
		\end{equation*}
Note that in view of \eqref{bilinear36}, the exponent of $N_{\min}$ is negative for $\epsilon$ sufficiently small and hence \eqref{bilinear19} holds.

		It remains to consider the case $N_{\min}=\min\{N_{1},N_{2}\}.$ But in this case we must have $L_{\max}=L_{j_{1}},$ with $j_{0}=j_{1}\in \{1,2\}.$
		From   \eqref{bilinear9}, \eqref{bilin52}, and the fact that $N_{\max}\sim\max\{N_1,N_2\}$, we obtain
		\begin{equation*}
			{
				\begin{split}
					I_{N_1,N_2,N_3}^{L_1,L_2,L_3}
					& \lesssim
					\frac{N_{3}^{1-\frac{\delta}{2}+\epsilon \delta } } 
					{N_{1}^{\frac{\delta}{2}}  N_{2}^{\frac{\delta}{2}} 	L_{\max}^{b} L_{\min}^{b} L_{med}^{1-b-\epsilon} }\\
					& \lesssim
					\frac{N_{\max}^{1-\delta+\epsilon \delta } } 
					{N_{\min}^{\frac{\delta}{2}}  	L_{\max}^{b} L_{\min}^{b} L_{med}^{1-b-\epsilon} }\\
					& \lesssim
					\frac{ N_{\max}^{1-\delta+\epsilon \delta-(2m-\delta)(1-b-\epsilon)}} 
					{  N_{\min}^{\frac{\delta}{2}+1-b-\epsilon} L_{\max}^{b} 	L_{\min}^{b}   
					}.		
			\end{split}}
		\end{equation*}
Since
$$
\frac{\delta}{2}+1-b-\epsilon=\frac{\delta}{2}+\frac{1}{2}+\epsilon+\left(\frac{1}{2}-b-2\epsilon\right)
$$	
and $\frac{1}{2}-b-2\epsilon<0$
	we may write
		\begin{equation*}
			{
				\begin{split}
					I_{N_1,N_2,N_3}^{L_1,L_2,L_3}
					& \lesssim
					\frac{ N_{\max}^{1-\delta+\epsilon \delta-(2m-\delta)(1-b+\epsilon)}N_{\min}^{-\frac{1}{2}+b+2\epsilon}} 
					{  N_{\min}^{\frac{\delta}{2}+\frac{1}{2}+\epsilon}  L_{\max}^{b} 	L_{\min}^{b}   
					}	\\
					& \lesssim
					\frac{ N_{\max}^{1-\delta+\epsilon \delta-(2m-\delta)(1-b+\epsilon)-\frac{1}{2}+b+2\epsilon}} 
					{  N_{\min}^{\frac{\delta}{2}+\frac{1}{2}+\epsilon} L_{\max}^{b} 	L_{\min}^{b}   
					}.
			\end{split}}
		\end{equation*}
	Note that the exponent of $N_{\max}$ is negative for $\epsilon$ small provided \eqref{bilinear36} holds. Therefore,
		\begin{equation*}
			{
				\begin{split}
						I_{N_1,N_2,N_3}^{L_1,L_2,L_3}
					& \lesssim
					\frac{1} 
					{  N_{\min}^{\frac{\delta}{2}} N_{\min}^{\frac{1}{2}+\epsilon} L_{\max}^{b-\epsilon} L_{\max}^{\epsilon} 	L_{\min}^{b}   	}
					=\frac{1} 
					{  N_{\min}^{\frac{1}{2}+\epsilon} N_{j_1}^{\frac{\delta}{2}} L_{j_1}^{b-\epsilon} L_{\max}^{\epsilon}  	},
			\end{split}}
		\end{equation*}
	which is \eqref{bilinear19} with $b'=b-\epsilon$.
This concludes the proof of the lemma in the case $N_{\max}\geq a.$

Finally, we analyse the case $N_{\max}\leq a.$ In this case, all frequencies are bounded and we do not need a careful analysis of the the power of $N_{\max}$ as before. Indeed, from \eqref{bilinear14.1} and the fact that $1-b-2\epsilon>0$, we obtain
		\begin{equation*}
			{
				\begin{split}
					\frac{N_{3}^{1-\delta \left(b-\frac{1}{2}\right)-1+b+\epsilon} } 
					{N_{1}^{\frac{\delta}{2}}  N_{2}^{\frac{\delta}{2}} 	L_{med}^{b} L_{\min}^{b} L_{\max}^{1-b-\epsilon}}
					&\leq
					\frac{N_{\max}^{-\delta \left(b-\frac{1}{2}\right)+b+\epsilon} } 
					{N_{j_{1}}^{\frac{\delta}{2}}   	L_{j_{1}}^{b}  L_{\max}^{1-b-2\epsilon} L_{\max}^{\epsilon}}\\
					& \leq
					\frac{N_{\max}^{-\delta \left(b-\frac{1}{2}\right)+b+\epsilon} N_{\min}^{\frac{1}{2}+\epsilon} } 
					{N_{\min}^{\frac{1}{2}+\epsilon} N_{j_{1}}^{\frac{\delta}{2}}   	L_{j_{1}}^{b}   L_{\max}^{\epsilon}}\\
					&\lesssim_{\delta,b,\epsilon}
					\frac{1 } 
					{N_{\min}^{\frac{1}{2}+\epsilon} N_{j_{1}}^{\frac{\delta}{2}}   	L_{j_{1}}^{b}   L_{\max}^{\epsilon}},	
			\end{split}}
		\end{equation*}
		where $j_{1}\in \{1,2\}.$
		This proves the lemma.
	\end{proof}
	
By invoking Lemma \ref{bilinear27}, we get that for some $\epsilon>0$ small, some $j_{1}\in \{1,2,3\}$ and some $b'\in \mathbb{R}$ with $\displaystyle{\frac{1}{2}<b'\leq b},$
	\begin{equation}\label{bilinear20}
		{\begin{split}
				I&\lesssim
				\sup_{\substack{w\in S_\epsilon^b} } \left(  \sum_{\substack{N_1,N_2,N_3\\
						L_1,L_2,L_3}} 
				\int\limits_{\Gamma }
				\frac{ \widehat{ P_{N_{2}}Q_{L_{2}} f}(k_{2},\tau_{2}) \widehat{P_{N_{1}}Q_{L_{1}}  g}(k_{1},\tau_{1})
					\widehat{ P_{N_{3}}Q_{L_{3}} h}(k_{3},\tau_{3})}
				{
					N_{\min}^{\frac{1}{2}+\epsilon} 
					N_{j_{1}}^{\frac{\delta}{2}} 
					L_{j_{1}}^{b'}	
					L_{\max}^{\epsilon}
				} dS \right).
		\end{split}}
	\end{equation}

In order to estimate the right-hand side of \eqref{bilinear20}, we claim that 
		\begin{equation}\label{bilinear24}
			{\begin{split}
					II&:= 
					\int\limits_{\Gamma }
					\frac{ \widehat{ P_{N_{2}}Q_{L_{2}} f}(k_{2},\tau_{2}) \widehat{P_{N_{1}}Q_{L_{1}}  g}(k_{1},\tau_{1})
						\widehat{ P_{N_{3}}Q_{L_{3}} h}(k_{3},\tau_{3})}
					{
						N_{\min}^{\frac{1}{2}+\epsilon} 
						N_{j_{1}}^{\frac{\delta}{2}} 
						L_{j_{1}}^{b'}	} dS\\
					&\lesssim \left\| \widehat{P_{N_{1}}Q_{L_{1}}  g}(k_{1},\tau_{1}) \right\|_{l^{2}_{k_{1}}L^{2}_{\tau_{1}}}		
					\left\|\widehat{ P_{N_{2}}Q_{L_{2}} f}(k_{2},\tau_{2})\right \|_{l^{2}_{k_{2}}L^{2}_{\tau_{2}}}  
					\left\|\widehat{ P_{N_{3}}Q_{L_{3}} h}(k_{3},\tau_{3})\right\|_{l^{2}_{k_{3}}L^{2}_{\tau_{3}}}.
			\end{split}}
		\end{equation}
Indeed, if $N_{\min}=N_{j_{0}}$ with $j_{0}\in \{1,2,3\}$ and $j_{0}\neq j_{1},$ then the proof is similar to that of CLAIM 7 in \cite{Flores OH and Smith}. Here, we show the case $j_{0}=j_{1}.$ Without loss of generality we can assume $j_{0}=j_{1}=1.$ By setting $\widehat{g^{\ast}}(k_{1},\tau_{1})= | k_{1}|^{-\frac{1}{2}-\epsilon} |k_{1}|^{-\frac{\delta}{2}} L_{1}^{-b'} \widehat{P_{N_{1}}Q_{L_{1}}  g}(k_{1},\tau_{1})$ and recalling that $k_1+k_2+k_3=0$ and $\tau_1+\tau_2+\tau_3=0$ we have 
		\begin{equation*}
				\begin{split}
					II&\lesssim  
					\int\limits_{\Gamma }
					\widehat{g^{\ast}}(k_{1},\tau_{1})   \widehat{ P_{N_{2}}Q_{L_{2}} f}(k_{2},\tau_{2})  \widehat{ P_{N_{3}}Q_{L_{3}} h}(k_{3},\tau_{3})  
					dS\\
					&=\sum_{k_{3},k_{2}\in \mathbb{Z}} \int\limits_{\mathbb{R}} \int\limits_{\mathbb{R}} 
					\widehat{g^{\ast}}(-k_{3}-k_{2},-\tau_{3}-\tau_{2})   \widehat{ P_{N_{2}}Q_{L_{2}} f}(k_{2},\tau_{2})  \widehat{ P_{N_{3}}Q_{L_{3}} h}(k_{3},\tau_{3})  
					d\tau_{2} d\tau_{3}\\
					&\leq \int\limits_{\mathbb{R}} \sum_{k_{3}\in \mathbb{Z}}
					\left|\left(\widehat{g^{\ast}} \ast  \widehat{ P_{N_{2}}Q_{L_{2}} f} \right)(-k_{3},-\tau_3)\right|   |\widehat{ P_{N_{3}}Q_{L_{3}} h}(k_{3},\tau_{3})|  
					d\tau_{3},
				\end{split}
			\end{equation*}
		where the convolution is taken over time and spatial variables.
	Then, using Cauchy-Schwarz and Young's inequalities
		\begin{equation}\label{bilinear57}
			{\begin{split}
					II
					&\lesssim \left\|
					 \widehat{g^{\ast}} \ast  \widehat{ P_{N_{2}}Q_{L_{2}} f}\right\|_{l^2_{k_3}L^{2}_{\tau_{3}}} \left\|\widehat{ P_{N_{3}}Q_{L_{3}} h}(k_{3},\tau_{3})\right\|_{l^{2}_{k_{3}}L^{2}_{\tau_{3}}}  \\
					&\leq \left\|
					\widehat{g^{\ast}}(k_{1},\tau_{1}) \right\|_{l^{1}_{k_{1}}L^{1}_{\tau_{1}}}  \left\|\widehat{ P_{N_{2}}Q_{L_{2}} f}(k_{2},\tau_{2})\right \|_{l^{2}_{k_{2}}L^{2}_{\tau_{2}}}  
					\left\|\widehat{ P_{N_{3}}Q_{L_{3}} h}(k_{3},\tau_{3})\right\|_{l^{2}_{k_{3}}L^{2}_{\tau_{3}}}.
			\end{split}}
		\end{equation}
Note that 	using Holder's inequality, 
\begin{equation}\label{bilinear58}
{
\begin{split}
\left\|
\widehat{g^{\ast}} \right\|_{l^{1}_{k_{1}}L^{1}_{\tau_{1}}} 
	&\leq 
\left\| \langle k_{1} \rangle^{-\frac{1}{2}-\epsilon}\right\|_{l^2_{k_1}} \sup_{k_1}\left\| \langle k_{1}\rangle^{-\frac{\delta}{2}} \left\langle \frac{\tau_{1}-\lambda_{k_{1}}}{\langle k_{1} \rangle^{\delta}} \right\rangle^{-b'} \right\|_{L^{2}_{\tau_{1}}}
\left\| \widehat{P_{N_{1}}Q_{L_{1}}  g}(k_{1},\tau_{1}) \right\|_{l^{2}_{k_{1}}L^{2}_{\tau_{1}}}\\
	&\lesssim_{\epsilon, \delta, b'}
\left\| \widehat{P_{N_{1}}Q_{L_{1}}  g}(k_{1},\tau_{1}) \right\|_{l^{2}_{k_{1}}L^{2}_{\tau_{1}}},	
\end{split}}
\end{equation}
where we have used that the first term on the right-hand side of \eqref{bilinear58} is clearly finite and for the second one we observe that a change of variables gives 
$$
\left\| \langle k_{1}\rangle^{-\frac{\delta}{2}} \left\langle \frac{\tau_{1}-\lambda_{k_{1}}}{\langle k_{1} \rangle^{\delta}} \right\rangle^{-b'} \right\|_{L^{2}_{\tau_{1}}}^2
= \int_{\mathbb{R}}\langle k_{1}\rangle^{-\delta} \left\langle \frac{\tau_{1}-\lambda_{k_{1}}}{\langle k_{1} \rangle^{\delta}} \right\rangle^{-2b'}d\tau_1=\int \langle\tau\rangle^{-2b'}d\tau,
$$
	which is finite because $b'>\frac{1}{2}$. Therefore, the claim follows from \eqref{bilinear57} and \eqref{bilinear58}.

With the above claim in hand, we are finally able to complete the proof of Theorem \ref{bilinear}. In fact, from \eqref{bilinear24} we deduce
	\begin{equation*}
		{\begin{split}
				I&\lesssim 
				\sup_{\substack{w\in S_\epsilon^b}} \left(  \sum_{\substack{N_1,N_2,N_3\\
						L_1,L_2,L_3}} 
				L_{\max}^{-\epsilon} \left\| \widehat{ g}(k_{1},\tau_{1}) \right\|_{l^{2}_{k_{1}}L^{2}_{\tau_{1}}}		
				\left\|\widehat{  f}(k_{2},\tau_{2})\right \|_{l^{2}_{k_{2}}L^{2}_{\tau_{2}}}  
				\left\|\widehat{  h}(k_{3},\tau_{3})\right\|_{l^{2}_{k_{3}}L^{2}_{\tau_{3}}}  \right)\\
				&\leq	\sup_{\substack{w\in S_\epsilon^b}} \left(  \sum_{\substack{N_1,N_2,N_3\\
						L_1,L_2,L_3}} 
				L_{\max}^{-\epsilon} \left\|  v \right\|_{Z^{b}}		
				\left\| u\right \|_{Z^{b}}  
				\left\| w\right\|_{Z^{1-b-\epsilon}}  \right)\\
				&\leq
				\left\|  v \right\|_{Z^{b}}		
				\left\| u\right \|_{Z^{b}}  
				\sum_{\substack{N_1,N_2,N_3\\
						L_1,L_2,L_3}} 
				\left(L_{\max}^{-\frac{\epsilon}{6}}\right)^{6}
		\end{split}}
	\end{equation*}
	Since we are summing over diadic indices and
	Lemma \ref{bilin19} (recall that $2m>\delta$) imply  that $L_{\max}$ dominates all other dyadic index (provided $N_{\max}$ is large), the factor $L_{\max}^{-\frac{\epsilon}{6}}$ makes each summation convergent. The proof of Theorem \ref{bilinear} is thus completed.
\end{proof}

\section{Global well-posedness}\label{GWP1}
In this section we will show our global well-posedness result stated in Theorem \ref{GWP}. As we have already observed in the introduction, it is a direct consequence of the following result.

\begin{thm}\label{Gwp}
	Let $\alpha>0,$ $\beta>0,$ $r>0,$ $\mu \in \mathbb{R},$ $m>\frac{1}{2},$ with $r<m$  and $T>0$ be given. Then for any $v_{0}\in L^{2}_{0}(\mathbb{T})$ and any $\delta\leq 1$ with $\max\{0, 2-2m\}<\delta$   the IVP \eqref{GB6} admits a unique solution $v\in C([0,T];L_{0}^{2}(\mathbb{T})).$ Moreover, the solution map $v_{0}\in L^{2}_{0}(\mathbb{T})\longmapsto v(t)\in C([0,T];L_{0}^{2}(\mathbb{T}))$ is uniformly continuous within a bounded set of $L^{2}_{0}(\mathbb{T}).$
\end{thm}
\begin{proof}
	First we will show the local well posedness of IVP \eqref{GB6} in $L_{0}^{2}(\mathbb{T})$, which in turn is equivalent to showing the local well-posedness of \eqref{GB7}. The strategy is  to prove that the operator 
	\begin{equation*}
		\Gamma(v):=S_{\mu}(t)v_{0}-\int_{0}^{t}S_{\mu}(t-s) \left( \partial_{x}(v^{2})+N_{1}v+Rv\right)(s)ds
	\end{equation*}
is a contraction in some ball
	$$
	B_{M}(S_{\mu}(\cdot)v_{0}):=\left\{ v\in Z_{T_{0}}^{b^{\ast}}: \;\|v-S_{\mu}(t)v_{0}\|_{Z_{T_{0}}^{b^{\ast}}}\leq M\right\},
	$$
where $0<T_0<1$ and $M>0$ are suitable constants to be chosen later and
$$
\frac{1}{2}<b^*<\min\left\{b,\frac{2m}{2m+\delta}\right\},
$$
with $b$ satisfying \eqref{minb}.
 By starting with $\displaystyle{M<\frac{1}{2}\|S_{\mu}(t)v_{0}\|_{Z_{T_{0}}^{b^{\ast}}} }$ it is easily seen that  ${\|v\|_{Z_{T_{0}}^{b^{\ast}}}\sim \|S_{\mu}(t)v_{0}\|_{Z_{T_{0}}^{b^{\ast}}}}$ for any $v\in B_{M}(S_{\mu}(\cdot)v_{0})$. In particular, Proposition \ref{locdamp6} implies that for any $v\in B_{M}(S_{\mu}(\cdot)v_{0})$ there exists a constant $C_{1}\equiv C_{1}(\delta,b,g)>0$ satisfying
	\begin{equation}\label{Lwposed1}
		\|v\|_{Z_{T_{0}}^{b^{\ast}}}\leq 2 \|S_{\mu}(t)v_{0}\|_{Z_{T_{0}}^{b^{\ast}}}\leq 2C_{1}\|v_{0}\|_{L^{2}_{0}(\mathbb{T})}.
	\end{equation}
From Theorems \ref{locdamp8},  \ref{bilinear}, and \ref{stiN1}, and Proposition \ref{stiR1}, we may found constants $\epsilon>0$ small and $C>0$ such that
	\begin{equation}\label{Lwposed2}
		{\begin{split}
				\left\|\Gamma(v)-S_{\mu}(t)v_{0} \right\|_{Z_{T_{0}}^{b^{\ast}}}
				&\lesssim
			\left\|D^{-\delta\left(b^{\ast} -\frac{1}{2}\right)}\partial_{x}(v^{2}) \right\|_{Z_{T_{0}}^{b^{\ast}-1}}
				+ \left\|D^{-\delta\left(b^{\ast} -\frac{1}{2}\right)}N_{1} \right\|_{Z_{T_{0}}^{b^{\ast}-1}}
				+
				\left\|D^{-\delta\left(b^{\ast} -\frac{1}{2}\right)}R(v) \right\|_{Z_{T_{0}}^{b^{\ast}-1}}\\
				&\leq
			CT_{0}^{\epsilon}\left\|v \right\|^{2}_{Z_{T_{0}}^{b^{\ast}}}
				+
			CT_{0}^{\epsilon}\left\|v \right\|_{Z_{T_{0}}^{b^{\ast}}}
				+
			CT_{0}^{\epsilon}\left\|v \right\|_{Z_{T_{0}}^{b^{\ast}}}\\
			&\leq 2CT_{0}^{\epsilon} \left\|v \right\|_{Z_{T_{0}}^{b^{\ast}}} 
			\left(\left\|v \right\|_{Z_{T_{0}}^{b^{\ast}}}
			+ 1\right).
		\end{split}}
	\end{equation}
	Hence, if $T_0$ is small enough such that
	\begin{equation}\label{Lwposed4}
		{
			\begin{split}
				2CT_{0}^{\epsilon}C_{1} \left\|v_{0} \right\|_{L_{0}^{2}(\mathbb{T})} 
				\left(4C_{1}\left\|v_{0} \right\|_{L_{0}^{2}(\mathbb{T})}
				+ 2\right)\leq M,
		\end{split}}
	\end{equation}
we deduce, from \eqref{Lwposed1}, 
$$
	\left\|\Gamma(v)-S_{\mu}(t)v_{0} \right\|_{Z_{T_{0}}^{b^{\ast}}}\leq 	2CT_{0}^{\epsilon}C_{1} \left\|v_{0} \right\|_{L_{0}^{2}(\mathbb{T})} 
	\left(4C_{1}\left\|v_{0} \right\|_{L_{0}^{2}(\mathbb{T})}
	+ 2\right)\leq M.
$$

	On the other hand,  for any $v_{1},v_{2}\in B_{M}(S_{\mu}(\cdot)v_{0})$, from computations similar to \eqref{Lwposed2}, we get
	\begin{equation*}
		{
			\begin{split}
				\left\|\Gamma(v_{1})-\Gamma(v_{2}) \right\|_{Z_{T_{0}}^{b^{\ast}}}
				&\leq
				CT_{0}^{\epsilon}\left\|v_{1}-v_{2} \right\|_{Z_{T_{0}}^{b^{\ast}}}
				\left\|v_{1}+v_{2} \right\|_{Z_{T_{0}}^{b^{\ast}}}
				+
				2CT_{0}^{\epsilon}\left\|v_{1}-v_{2} \right\|_{Z_{T_{0}}^{b^{\ast}}}\\
				&\leq
				2CT_{0}^{\epsilon}
				\left(\left\|v_{1} \right\|_{Z_{T_{0}}^{b^{\ast}}}
				+	\left\|v_{2} \right\|_{Z_{T_{0}}^{b^{\ast}}}
				+ 1\right)
				\left\|v_{1}-v_{2} \right\|_{Z_{T_{0}}^{b^{\ast}}} 
				.
		\end{split}}
	\end{equation*}
	From \eqref{Lwposed1}, \eqref{Lwposed4}  and the fact that $\displaystyle{M<\frac{1}{2}\|S_{\mu}(t)v_{0}\|_{Z_{T_{0}}^{b^{\ast}}} }$ we obtain $	2CT_{0}^{\epsilon}
	\left(4C_{1}\left\|v_{0} \right\|_{L_{0}^{2}(\mathbb{T})}
	+ 1\right)\leq \frac{1}{2}$. Thus,
	\begin{equation*}
		{
			\begin{split}
				\left\|\Gamma(v_{1})-\Gamma(v_{2}) \right\|_{Z_{T_{0}}^{b^{\ast}}}
				&\leq
				2CT_{0}^{\epsilon}
				\left(4C_{1}\left\|v_{0} \right\|_{L_{0}^{2}(\mathbb{T})}
				+ 1\right)
				\left\|v_{1}-v_{2} \right\|_{Z_{T_{0}}^{b^{\ast}}}\\
				&\leq
				\frac{1}{2}
				\left\|v_{1}-v_{2} \right\|_{Z_{T_{0}}^{b^{\ast}}}.	
		\end{split}}
	\end{equation*}
	This shows that $\Gamma: B_{M}(S_{\mu}(\cdot)v_{0}) \to B_{M}(S_{\mu}(\cdot)v_{0})$ is a contraction map for $T_{0}\equiv T_{0}(\|v_{0}\|_{L_{0}^{2}(\mathbb{T})})$ satisfying \eqref{Lwposed4}. From Proposition \ref{prop2}, we infer that its unique fixed point $v$ belongs to $C([0,T];L_{0}^{2}(\mathbb{T})).$ This proves the local-well posedness of IVP \eqref{GB6}. 
	
	To show that this local solution can be extended to any time interval of the form $[0,T]$ we note that solutions of \eqref{GB6} satisfies 
	\begin{equation}\label{energystimate}
		\frac{1}{2}\frac{d}{dt}\left(\|v(\cdot,t)\|^2_{L_{0}^{2}(\mathbb{T})}\right)=-\|D^{\frac{\delta}{2}}Gv(\cdot,t)\|^2_{L_{0}^{2}(\mathbb{T})}
	\end{equation}
implying that $\displaystyle{\|v(\cdot,t)\|_{L_{0}^{2}(\mathbb{T})}}$ is decreasing in the temporal variable $t\geq0.$ Therefore,
	$\|v(\cdot,t)\|_{L_{0}^{2}(\mathbb{T})}\leq \|v_{0}\|_{L_{0}^{2}(\mathbb{T})}$
	and we can repeat the above argument with a uniform-size local time interval.

The uniform continuity of the map data-to-solution follows in a standard way. The proof of the theorem is thus complete.

\end{proof}

\section{exponential stabilization}\label{LES}

This section is devoted to prove our exponential stabilization result. As before, throughout the section we assume $\mu\in \mathbb{R},$ $\alpha>0,$ $\beta>0,$ $0<\delta\leq1,$ $m>\frac{1}{2},$ and $0<r<m.$

\subsection{Linear Stabilization}\label{linstab}
In this section, we state a linear stabilization result which is fundamental to prove the local exponential stabilization result presented in Theorem \ref{Sta1}. We consider the equation
\begin{equation}\label{Linestab1}
	\begin{cases}
		\partial_{t}v+\beta D^{2m}\partial_{x}v+\alpha \mathcal{H}^{2r}\partial_{x}v+2\mu  \partial_{x}v=-GD^{\delta}Gv, \quad  t>0,\\
		v(x,0)=v_{0}(x), 
	\end{cases}
\end{equation}
where $v_{0}\in L_{0}^{2}(\mathbb{T}).$

As in Section \ref{Linear Systems}, let $\mathcal{A}$ denote the operator  $-\beta D^{2m}- \alpha H^{2r}-2\mu$. Since $G$ defined in \eqref{EQ1} is self-adjoint on $L_{0}^{2}(\mathbb{T}),$ from \eqref{skew} it is easy to see that, for any $v\in H_{0}^{2m+1}(\mathbb{T})$,
$$
\left((\partial_{x}\mathcal{A} -GD^{\delta}G)v,v\right)_{L_{0}^{2}(\mathbb{T})}= 
\left((\partial_{x}\mathcal{A} -GD^{\delta}G)^{\ast}v,v\right)_{L_{0}^{2}(\mathbb{T})}
=-\|D^{\frac{\delta}{2}}Gv \|^{2}_{L_{0}^{2}(\mathbb{T})}.
$$
This means that both $\partial_{x}\mathcal{A}-GD^{\delta}G$ and $(\partial_{x}\mathcal{A} -GD^{\delta}G)^*$ are dissipative in $L_{0}^{2}(\mathbb{T})$ (see \cite[Definition 4.1, page 13]{Pazy}). Hence, we conclude that $\partial_{x}\mathcal{A} -GD^{\delta}G$ is the infinitesimal generator of a $C_{0}$-semigroup of contractions on $L_{0}^{2}(\mathbb{T})$ (see \cite[Corollary 4.4, page 14]{Pazy}). We denote this semigroup by $\{	W(t)\}_{t\geq0}$. Actually, we will show that $W(t)$ has an exponential decay.
The following estimate will be needed.
\begin{lem}\label{staop2}
	Let $b^{\ast}\in \mathbb{R}$ be given with $\displaystyle{\frac{1}{2}<b^{\ast}<\min\left\{\frac{2m}{2m+\delta},b\right\}},$ where $b$ satisfies \eqref{minb}. Then, for any $v_{0}\in L_{0}^{2}(\mathbb{T})$ and any $T>0$
	$$\|W(t)v_{0}\|_{Z_{T}^{b^{\ast}}}\lesssim_{b^{\ast},T}\|v_{0}\|_{L_{0}^{2}(\mathbb{T})}.$$
\end{lem}
\begin{proof}
We already know that
	\begin{equation}\label{staop4}
		v(t)=W(t)v_{0}
	\end{equation}
	is the unique solution of \eqref{Linestab1}. On the other hand, using the decomposition \eqref{locdamp} and Duhamel's formula we see that it can also be written as
	\begin{equation}\label{staop5}
		v(t)=S_{\mu}(t)v_{0}-\int_{0}^{t}S_{\mu}(t-s)(N_{1}v+Rv)(s)ds
	\end{equation}
	Substituting \eqref{staop4} in \eqref{staop5}, we obtain
	$$W(t)v_{0}=S_{\mu}v_{0}-\int_{0}^{t}S_{\mu}(t-s)(N_{1}W(\cdot)v_{0}+RW(\cdot)v_{0})(s)ds.$$
	Let $0<T_0<1$ be a small real number to be chosen later. Using Proposition \ref{locdamp6} and similar computations as in \eqref{Lwposed2}, we deduce
	$$
	{\begin{split}
			\left\|W(t)v_{0}\right\|_{Z_{T_{0}}^{b^{\ast}}}
			&\leq
			\left\|S_{\mu}(t)v_{0}\right\|_{Z_{T_{0}}^{b^{\ast}}}
			+
			\left\|\int_{0}^{t}S_{\mu}(t-s)(N_{1}W(\cdot)v_{0}+RW(\cdot)v_{0})(s)ds\right\|_{Z_{T_{0}}^{b^{\ast}}}\\
			&\leq
			C\|v_{0}\|_{L_{0}^{2}(\mathbb{T})}
			+
			CT_0^{\epsilon} \left\|W(t)v_{0}\right\|_{Z_{T_{0}}^{b^{\ast}}}
			+
			CT_{0}^{\epsilon} \left\|W(t)v_{0}\right\|_{Z_{T_{0}}^{b^{\ast}}}\\
			&\leq 	C\|v_{0}\|_{L_{0}^{2}(\mathbb{T})}+2CT_0^{\epsilon} \left\|W(t)v_{0}\right\|_{Z_{T_{0}}^{b^{\ast}}}.
	\end{split}}
$$
By	choosing $T_0$ satisfying $\displaystyle{2CT_{0}^{\epsilon}<1},$ we obtain
	$$
	{
		\begin{split}
			\left\|W(t)v_{0}\right\|_{Z_{T_{0}}^{b^{\ast}}}
			&\lesssim\|v_{0}\|_{L_{0}^{2}(\mathbb{T})}.
	\end{split}}
	$$
Since $T_{0}$ is an absolute constant, we can iterate the above argument on uniform-size intervals to obtain the desired.
\end{proof}

Next, we establish an exponential stabilization result for system \eqref{Linestab1}.
\begin{prop}\label{staop7}
	There exist $\gamma>0$ and $M>0$ such that for any $v_{0}\in L_{0}^{2}(\mathbb{T})$ the unique solution $v(t)=W(t)v_{0}$ of \eqref{Linestab1} satisfies
	\begin{equation}\label{staop12}	
		\left\|W(t)v_{0}\right\|_{L_{0}^{2}(\mathbb{T})}
		\leq M e^{-\gamma t} \left\|v_{0}\right\|_{L_{0}^{2}(\mathbb{T})},\;\;\;\text{for all}\;t\geq 0.	
	\end{equation}
\end{prop}
\begin{proof}
In view of \eqref{energystimate}, we obtain
	\begin{equation}\label{staop9}
		\frac{1}{2}\left\|v(\cdot,t) \right\|^{2}_{L_{0}^{2}(\mathbb{T})}
		=
			\frac{1}{2}\left\|v_{0} \right\|^{2}_{L_{0}^{2}(\mathbb{T})}
		-
		\int_{0}^{t} \left\|D^{\frac{\delta}{2}}Gv(\cdot,t') \right\|^{2}_{L_{0}^{2}(\mathbb{T})} dt', \quad t\geq0.
	\end{equation}
	Hence, to prove the proposition it is sufficient to establish the following \textit{linear observability inequality}: there exist $T>0$ and  $C>2$ such that, for any $v_{0}\in L_{0}^{2}(\mathbb{T}),$ 
	\begin{equation}\label{staop10}
		\left\|v_{0} \right\|^{2}_{L_{0}^{2}(\mathbb{T})}
		\leq C
		\int_{0}^{T} \left\|D^{\frac{\delta}{2}}Gv(\cdot,t') \right\|^{2}_{L_{0}^{2}(\mathbb{T})} dt'.
	\end{equation}
	Indeed, if \eqref{staop10} holds then from \eqref{staop9}, we have
	$$\left\|v(\cdot,T) \right\|^{2}_{L_{0}^{2}(\mathbb{T})}
	\leq
	\rho \left\|v_{0} \right\|^{2}_{L_{0}^{2}(\mathbb{T})},
	$$
	for some $0<\rho<1$. We can repeat this estimate on successive intervals $[(l-1)T,lT]$ to get
	$$\left\|v(\cdot,lT) \right\|^{2}_{L_{0}^{2}(\mathbb{T})}
	\leq
	\rho^{l} \left\|v_{0} \right\|^{2}_{L_{0}^{2}(\mathbb{T})},\;\;\;l=2,3,\ldots,
	$$
which in turn, from the semigroup properties, implies \eqref{staop12}.
	
	Now we prove \eqref{staop10}. 	We argue by contradiction assuming that \eqref{staop10} does not hold. Then for any $n\in \mathbb{N}^{\ast},$ with $n>1$, we can find a sequence $u_{n}=W(t)(u_{n}(0))$ of solutions of \eqref{Linestab1} such that (after normalization)
	\begin{equation*}
		u_{n}\in Z_{T}^{b^{\ast}}\cap C([0,T];L_{0}^{2}(\mathbb{T})),
	\end{equation*}
	\begin{equation*}
		\left\|u_{n}(0) \right\|_{L_{0}^{2}(\mathbb{T})}=1,
	\end{equation*}
	\begin{equation}\label{staop16}
		\int_{0}^{T} \left\|D^{\frac{\delta}{2}}Gu_{n}(\cdot,t') \right\|^{2}_{L_{0}^{2}(\mathbb{T})} dt'<\frac{1}{n}.
	\end{equation}
with $b^*$ as in Lemma \ref{staop2}. 

Next we introduce the negative number   $\displaystyle{\gamma^{\ast}=\frac{\delta}{2}-(2m+1)}$. Since the operator $\partial_{x}\mathcal{A}$ is bounded from $H_{0}^{\frac{\delta}{2}}(\mathbb{T})$ into $H^{\gamma^*}_0(\mathbb{T})$, from  Proposition \ref{prop2}-(iii) and Lemma \ref{staop2},
	\begin{equation}\label{staop18}
		\begin{split}
				\left\|\partial_{x}\mathcal{A} u_{n} 
			\right\|_{L^{2}(0,T;H_{0}^{\gamma^{\ast}}(\mathbb{T} ))}
			\lesssim
			\left\| u_{n}\right\|_{L^{2} (0,T;H_{0}^{\frac{\delta}{2}}(\mathbb{T} ))}
			\lesssim \left\|u_{n} \right\|_{Z_{T}^{b^{\ast}}}
			\lesssim 	\left\|u_{n}(0) \right\|_{L_{0}^{2}(\mathbb{T})}\lesssim 1.
		\end{split}
	\end{equation}
 Also, because $G$ is bounded on $H_{0}^{s}(\mathbb{T})$ for any $s\in \mathbb{R}$ (see \cite[Lemma 2.20]{micu}), we infer 
	\begin{equation*}
		{
			\begin{split}
				\left\|GD^{\delta}Gu_{n} \right\|_{H_{0}^{\gamma^*}(\mathbb{T})}
				 \lesssim
				\left\|D^{\delta}Gu_{n} \right\|_{H_{0}^{\gamma^*}(\mathbb{T})}
				 \lesssim
				\left\|Gu_{n} \right\|_{H_{0}^{\gamma^{\ast}+\delta}(\mathbb{T})}
				 \lesssim
				\left\|Gu_{n} \right\|_{H_{0}^{\frac{\delta}{2}}(\mathbb{T})}.  
		\end{split}}
	\end{equation*}
	From this and the energy estimate \eqref{staop9}, we have
	\begin{equation}\label{staop24}
		{
			\begin{split}
				\left\|GD^{\delta}Gu_{n} \right\|^{2}_{L^{2}(0,T;H_{0}^{\gamma^{\ast}}(\mathbb{T}))}
						& \lesssim 
				\int_{0}^{T}\left\|Gu_{n} (t')\right\|^{2}_{H_{0}^{\frac{\delta}{2}}(\mathbb{T})}dt'\\
				&=\frac{1}{2}
			 \left( \left\|u_{n}(0) \right\|^{2}_{L_{0}^{2}(\mathbb{T})} - \left\|u_{n}(T) \right\|^{2}_{L_{0}^{2}(\mathbb{T})} \right)\\
				&\lesssim
			 \left\|u_{n}(0) \right\|^{2}_{L_{0}^{2}(\mathbb{T})} \\
				&\lesssim 1.
		\end{split}}
	\end{equation}
 Using the equation in \eqref{Linestab1}, \eqref{staop18} and \eqref{staop24}, we infer that
	\begin{equation}\label{staop25}
		{
			\begin{split}
				\left\|\partial_{t}u_{n} \right\|_{L^{2}(0,T;H_{0}^{\gamma^{\ast}}(\mathbb{T}))}
				&\leq
				\left\|\partial_{x}\mathcal{A}u_{n} \right\|_{L^{2}(0,T;H_{0}^{\gamma^{\ast}}(\mathbb{T}))}
				+
				\left\|GD^{\delta}Gu_{n} \right\|_{L^{2}(0,T;H_{0}^{\gamma^{\ast}}(\mathbb{T}))}
				\lesssim 1.
		\end{split}}
	\end{equation}

	Uniform bounds \eqref{staop18} and \eqref{staop25} allow us to apply the Aubin-Lions lemma (see, for instance, \cite[Section 7.3]{Roubicek}). Therefore, one can extract a subsequence (still denoted by $u_{n}$) with the following properties:
	\begin{equation}\label{staop26}
		u_{n}\to u \;\;\text{in}\;L^{2}(0,T;H_{0}^{\gamma}(\mathbb{T})),\;\text{as}\;n\to \infty, \;\text{for all}\; \gamma\;\text{with}\;\frac{\delta}{2}>\gamma\geq\gamma^{\ast},
	\end{equation}
	\begin{equation}\label{staop27}
		u_{n}\rightharpoonup u \;\;\text{in}\;L^{2}(0,T;H_{0}^{\frac{\delta}{2}}(\mathbb{T})),\;\text{as}\;n\to \infty,\qquad\qquad\qquad \qquad\qquad \quad
	\end{equation}
	for some $u\in L^{2}(0,T;H_{0}^{\frac{\delta}{2}}(\mathbb{T})).$
	On the other hand, the facts that $u_n\in C([0,T];L_{0}^{2}(\mathbb{T}))$ is the unique solution of \eqref{Linestab1} and that $W(t)$ is a $C_{0}$-semigroup of contractions on $L_{0}^{2}(\mathbb{T})$ imply
	\begin{equation}\label{staop28}
		\left\| u_{n}\right\|_{L^{\infty} ([0,T];L_{0}^{2}(\mathbb{T} ))}
		= 
		\left\|W(t)(u_{n}(0)) \right\|_{L^{\infty} ([0,T];L_{0}^{2}(\mathbb{T} ))}
		\lesssim
		\left\|u_{n}(0) \right\|_{L_{0}^{2}(\mathbb{T})}
		\lesssim1.
	\end{equation}
	 Hence, applying the Banach-Alaoglu-Bourbaki theorem (see, for instance, \cite[Theorem 3.16]{Brezis}) we can extract a subsequence (still denoted by $u_{n}$) satisfying in addition to \eqref{staop26} and \eqref{staop27} the following property:
	\begin{equation}\label{staop29}
		u_{n} \stackrel{\ast}{ \rightharpoonup} u \;\;\text{in}\;L^{\infty}(0,T;L_{0}^{2}(\mathbb{T})),\;\text{as}\;n\to \infty.
	\end{equation}
	
	Next, we shall prove that $\left\{u_{n}(0)\right\}_{n>1}$ is a Cauchy sequence in $L_{0}^{2}(\mathbb{T}).$ In fact, first note if $w$ is a solution of \eqref{Linestab1} then
multiplying the energy estimate \eqref{energystimate} by $(T-t)$ and integrating on the interval $[0,T]$ results
	$$
	\left\|w(0) \right\|^{2}_{L_{0}^{2}(\mathbb{T})}
	=\frac{1}{T} \int_{0}^{T}  \left\|w(\cdot,t) \right\|^{2}_{L_{0}^{2}(\mathbb{T})} dt
	+
	\frac{2}{T} \int_{0}^{T}(T-t) \left\|D^{\frac{\delta}{2}}Gw(\cdot,t) \right\|^{2}_{L_{0}^{2}(\mathbb{T})}dt.
	$$
	Using this last identity to the difference of two solutions $u_n-u_l$, in view of \eqref{staop26} and \eqref{staop16}, we get
	$$
	{
		\begin{split}
			\left\|u_{n}(0)-u_{l}(0) \right\|^{2}_{L_{0}^{2}(\mathbb{T})}
			&=\frac{1}{T} \int_{0}^{T}  \left\|(u_{n}-u_{l})(t) \right\|^{2}_{L_{0}^{2}(\mathbb{T})} dt
			+
			\frac{2}{T} \int_{0}^{T}(T-t) \left\|D^{\frac{\delta}{2}}G(u_{n}-u_{l})(t) \right\|^{2}_{L_{0}^{2}(\mathbb{T})}dt\\
			&\leq
			\frac{1}{T}  \left\|u_{n}-u_{l} \right\|^{2}_{L^{2}(0,T;L_{0}^{2}(\mathbb{T}))} 
			+
			2\int_{0}^{T} \left\|D^{\frac{\delta}{2}}G(u_{n}-u_{l})(t) \right\|^{2}_{L_{0}^{2}(\mathbb{T})}dt\\
			&<
			\frac{1}{T}  \left\|u_{n}-u_{l} \right\|^{2}_{L^{2}(0,T;L_{0}^{2}(\mathbb{T}))} 
			+
			4\left(\frac{1}{n}+\frac{1}{l}\right)
			\longrightarrow 0,\;\;\text{as}\;n,l \to \infty.
	\end{split}}
	$$
	Thus, $u_{n}(0)$ converges strongly to some $u_{0}$ in $L_{0}^{2}(\mathbb{T})$. From the continuous dependence it follows that the solution of equation \eqref{Linestab1} with initial data $u_{0}$, say,  $\tilde{u}$,  agrees with the limit of $u_{n}$ in the space $C([0,T];L_{0}^{2}(\mathbb{T})).$ But, from   \eqref{staop29} we have that $\tilde{u}\equiv u$ in $L^{\infty}(0,T;L_{0}^{2}(\mathbb{T})).$ Since $\tilde{u}\in C([0,T];L_{0}^{2}(\mathbb{T}))$ we conclude that $u\in C([0,T];L_{0}^{2}(\mathbb{T}))$ and $u(0)=u_{0}.$ In addition, in view of the weak convergence \eqref{staop27} and \eqref{staop16} we deduce
	$$
	\int_{0}^{T}\left\|D^{\frac{\delta}{2}}Gu(t) \right\|^{2}_{L_{0}^{2}(\mathbb{T})}dt\leq \liminf_{n\to\infty}\int_{0}^{T}\left\|D^{\frac{\delta}{2}}Gu_{n}(t) \right\|^{2}_{L_{0}^{2}(\mathbb{T})}dt=0,
	$$
	which gives that for almost every $t\in (0,T),$ $Gu(t)\in L^{1}(\mathbb{T})$ and 
	$\widehat{Gu}(k,t)=0,$ for all $k\in \mathbb{Z}^{\ast}.$
	Hence, $Gu(x,t)=0,$ a.e. $(x,t)\in \mathbb{T}\times (0,T).$ Using the definition of $G$ in $L_{0}^{2}(\mathbb{T}),$ we have that
	$$u(x,t)=\int_{\mathbb{T}}u(y,t)g(y)dy=:c(t), \;\;\text{a.e.}\;(x,t)\in \omega\times (0,T).$$
	Applying Cauchy-Schwarz's inequality, we note that \eqref{staop28} yields
	$$
	{
		\begin{split}
			\left\|c(t)\right\|_{L^{\infty}(0,T)}
			&\leq
			\sup\limits_{t\in (0,T)} \int_{\mathbb{T}} |u||g|dy
			& \leq
			\|g\|_{L_{0}^{2}(\mathbb{T})} \|u\|_{L^{\infty}(0,T; L_{0}^{2}(\mathbb{T}))}< \infty.
	\end{split}}
	$$
	Thus, the limit $u\in C([0,T];L_{0}^{2}(\mathbb{T}))$ satisfies
	\begin{equation}\label{Linestab2}
		\left \{
		\begin{array}{l l}
			\partial_{t}u+\beta D^{2m}\partial_{x}u+\alpha \mathcal{H}^{2r}\partial_{x}u+2\mu  \partial_{x}u=0,& x\in \mathbb{T},\; t>0,\\
			u(x,t)=c(t), &\text{for a.e.} \;(x,t)\in \omega \times (0,T).
		\end{array}
		\right.
	\end{equation}
	Derivating \eqref{Linestab2} with respect to  the spacial variable and setting $w:=\partial_{x}u \in C([0,T];H_{0}^{-1}(\mathbb{T})),$ we have that
	\begin{equation*}
		\left \{
		\begin{array}{l l}
			\partial_{t}w+\beta D^{2m}\partial_{x}w+\alpha \mathcal{H}^{2r}\partial_{x}w+2\mu  \partial_{x}w=0,& x\in \mathbb{T},\; t>0,\\
			w(x,t)=0, &\text{for a.e.} \;(x,t)\in \omega \times (0,T).
		\end{array}
		\right.
	\end{equation*}
	Finally, Proposition \ref{UCPLgBe} imply that $\partial_{x}u=w\equiv 0$ a.e. $\mathbb{T}\times (0,T).$
	Hence, for a.e. $t\in (0,T),$ 
	$$u(\cdot,t)=c_{1}(t),\;\;\text{for a.e.}\;x\in \mathbb{T}.
	$$
	Since $[u]=0,$ we conclude that $c_{1}(t)=0,\;\text{a.e.}\; t\in(0,T). $ Therefore, $u\equiv 0$ for a.e. $(x,t)\in \mathbb{T} \times (0,T).$
	This contradicts the fact that 
	$$
	\left\|u(0) \right\|_{L_{0}^{2}(\mathbb{T})}=
	\lim_{n\to \infty}	\left\|u_{n}(0) \right\|_{L_{0}^{2}(\mathbb{T})}
	=1.
	$$
	This proves the proposition.
\end{proof}

\subsection{Local exponential stabilization}\label{LES1}
In this section we show the local exponential stabilization result for  \eqref{GB3} in $L_{p}^{2}(\mathbb{T})$ stated in Theorem  \ref{Sta1}.
The following extension of the bilinear estimate is needed.
\begin{lem}\label{LER}
	Let $b^{\ast}\in \mathbb{R}$ be given with $\displaystyle{\frac{1}{2}<b^{\ast}<\min\left\{\frac{2m}{2m+\delta},b\right\}},$ where $b$ satisfies \eqref{minb}.
 Assume that $2m>2-\delta.$  Then, for any $T>0$,
	$$	\left\| \int_{0}^{t}W(t-s) \partial_{x}(v^{2})(s) ds\right\|_{Z_{T}^{b^{\ast}}}
	\lesssim_{T}
	\left\|v\right\|^{2}_{Z_{T}^{b^{\ast}}}
	$$
\end{lem}
\begin{proof}
The proof is similar to that of Lemma 5 in \cite{Flores OH and Smith}	(see also \cite[Lemma 4.4]{14}), so we omit the details.
\end{proof}

Theorem  \ref{Sta1} is a direct consequence of the following result.

\begin{thm}\label{LER2}
Assume $\max\{0, 2-2m\}<\delta\leq1$ (and therefore $\delta<2m$).  Then there exist $\rho>0$ and $\lambda'>0$ such that for any $v_{0}\in L^{2}_{0}(\mathbb{T})$ with $\|v_{0}\|_{L_{0}^{2}(\mathbb{T})}<\rho$ the unique solution $v\in C([0,+\infty);L_{0}^{2}(\mathbb{T}))$ of system \eqref{GB6} satisfies
	$$\|v(\cdot, t)\|_{L^{2}_{0}(\mathbb{T})}\leq Me^{-\lambda' t}\|v_{0}\|_{L^{2}_{0}(\mathbb{T})},$$	
for all $t\geq0$ and some positive constant $M$.
\end{thm}

\begin{proof}
From Proposition \ref{staop7}	we can fix some positive $T$ large enough  and some $\lambda'$ small enough with $0<\lambda'<\gamma$ such that
	\begin{equation}\label{LER4}
		\left\|W(T)v_{0}\right\|_{L_{0}^{2}(\mathbb{T})}
		\leq
		\frac{1}{2}e^{-\lambda' T}	\left\|v_{0}\right\|_{L_{0}^{2}(\mathbb{T})}.
	\end{equation}

As in the proof of Theorem \ref{Gwp} the idea is to show if $\|v_{0}\|_{L_{0}^{2}(\mathbb{T})}<\rho$ then the map 
	$$\Gamma(v)=W(t)v_{0}-\int_{0}^{t}W(t-s)(\partial_{x}(v^{2}))(s)ds
	$$ 
is a contraction in the ball 
	$$B_{M}(W(\cdot)v_{0}):=\left\{ v\in Z_{T}^{b^{\ast}}: \;\|v-W(t)v_{0}\|_{Z_{T}^{b^{\ast}}}\leq M\right\},$$
	for some suitable $M$ and $b^*$ as in Lemma \ref{staop2}.
	This will be achieved provided that $ \rho$ and $M$ are small enough. Furthermore, to ensure the exponential stability with the claimed decay rate, the numbers $M$ and $\rho$ will be chosen in such a way that
	\begin{equation}\label{LER12}
		{
			\begin{split}
				\left\|v(T) \right\|_{L_{0}^{2}(\mathbb{T})}
				&\leq
				e^{-\lambda' T}	\left\|v_{0}\right\|_{L_{0}^{2}(\mathbb{T})}.
		\end{split}}
	\end{equation}
	
So let us start by fixing $M>0$ such that $\displaystyle{M<\frac{1}{2}\|W(t)v_{0}\|_{Z_{T}^{b^{\ast}}} }$. It is easy to check that for any $v\in B_{M}(S_{\mu}(\cdot)v_{0})$ we have $\displaystyle{\|v\|_{Z_{T}^{b^{\ast}}}\sim \|W(t)v_{0}\|_{Z_{T}^{b^{\ast}}}}$ and (from Lemma \eqref{staop2}) there exists $C_{1}= C_{1}(b^{\ast},T)>0$ satisfying
	\begin{equation}\label{LER5}
		\|v\|_{Z_{T}^{b^{\ast}}}\leq 2 \|W(t)v_{0}\|_{Z_{T}^{b^{\ast}}}\leq 2C_{1}\|v_{0}\|_{L^{2}_{0}(\mathbb{T})}.
	\end{equation}
	Applying Lemma \ref{LER} and \eqref{LER5}, we get a positive constants $C_{2}= C_{2}(T)$ such that
	\begin{equation*}
		{
			\begin{split}
				\left\|\Gamma(v)-W(t)v_{0} \right\|_{Z_{T}^{b^{\ast}}}
				\leq
				C_{2}\left\|v \right\|^{2}_{Z_{T}^{b^{\ast}}}
				\leq
				4C_{2}C^{2}_{1}\left\|v_{0} \right\|^{2}_{L_{0}^{2}(\mathbb{T})}
				<
				4C_{2}C^{2}_{1}\rho^{2}  			
		\end{split}}
	\end{equation*}
and,  for any $v_{1},v_{2}\in B_{M}(W(\cdot)v_{0})$,
	\begin{equation*}
	{
		\begin{split}
			\left\|\Gamma(v_{1})-\Gamma(v_{2}) \right\|_{Z_{T}^{b^{\ast}}}
			&\leq
			C_{2}\left\|v_{1}-v_{2} \right\|_{Z_{T}^{b^{\ast}}}
			\left\|v_{1}+v_{2} \right\|_{Z_{T}^{b^{\ast}}}\\
			&\leq
			C_{2}\left(\left\|v_{1} \right\|_{Z_{T}^{b^{\ast}}}
			+	\left\|v_{2} \right\|_{Z_{T}^{b^{\ast}}}\right)
			\left\|v_{1}-v_{2} \right\|_{Z_{T}^{b^{\ast}}} \\
			&\leq 
			4C_{2}C_{1}\left\|v_{0} \right\|_{L_{0}^{2}(\mathbb{T})}
			\left\|v_{1}-v_{2} \right\|_{Z_{T}^{b^{\ast}}}\\ 	
			&<
			4C_{2}C_{1}\rho
			\left\|v_{1}-v_{2} \right\|_{Z_{T}^{b^{\ast}}}.
	\end{split}}
\end{equation*}
By choosing $\rho>0$ small enough such that $4C_2C_1^2\rho^2\leq M$, which also gives $4C_2C_1\rho<1/2$, we deduce that $\Gamma$ is a contraction in $B_{M}(W(\cdot)v_{0})\subset Z_{T}^{b^{\ast}}.$ Proposition \ref{prop2} implies that its unique fixed point, say,
	$v$, belongs to $C([0,T];L_{0}^{2}(\mathbb{T}))$. Finally, from \eqref{LER4} and Lemma \ref{LER} we infer that $v$
	fulfills
	\begin{equation*}
		{
			\begin{split}
				\left\|v(T) \right\|_{L_{0}^{2}(\mathbb{T})}
				&\leq
				\frac{1}{2}e^{-\lambda' T}	\left\|v_{0}\right\|_{L_{0}^{2}(\mathbb{T})}
				+
				\left\|\int_{0}^{t}W(t-s)(\partial_{x}(v^{2}))(s)ds \right\|_{C(0,T;L_{0}^{2}(\mathbb{T}))}
				\\
				&\leq
				\frac{1}{2}e^{-\lambda' T}	\left\|v_{0}\right\|_{L_{0}^{2}(\mathbb{T})}
				+
				C_{3}\left\|\int_{0}^{t}W(t-s)(\partial_{x}(v^{2}))(s)ds \right\|_{Z_{T}^{b^{\ast}}}
				\\
				&\leq
				\frac{1}{2}e^{-\lambda' T}	\left\|v_{0}\right\|_{L_{0}^{2}(\mathbb{T})}
				+
				C_{3}C_{2}4C^{2}_{1}\left\|v_{0} \right\|^{2}_{L_{0}^{2}(\mathbb{T})}\\
				&<
				\frac{1}{2}e^{-\lambda' T}	\left\|v_{0}\right\|_{L_{0}^{2}(\mathbb{T})}
				+
				C_{3}C_{2}4C^{2}_{1} \rho\left\|v_{0} \right\|_{L_{0}^{2}(\mathbb{T})},
		\end{split}}
	\end{equation*}
	for some positive constant $C_{3}$ depending on $b^{\ast}$ and $T$ given by Proposition \ref{prop2}. Thus, if additionally we choose $\rho$ satisfying 
	$2 C_{3}C_{2}4C_{1}^{2}\rho<e^{-\lambda' T},$
	we obtain that \eqref{LER12} holds. Using induction we can show that 
	\begin{equation*}
		{
			\begin{split}
				\left\|v(nT) \right\|_{L_{0}^{2}(\mathbb{T})}
				&\leq
				e^{-\lambda' nT}	\left\|v_{0}\right\|_{L_{0}^{2}(\mathbb{T})},\;\;\text{for any}\;n\geq 0\;\text{and some fixed}\;T>0,
		\end{split}}
	\end{equation*}
which completes the proof of the theorem.
\end{proof}


\subsection*{Acknowledgment}
F.J.V.L is supported by FAPESP/Brazil  grant 2020/14226-4.
A.P. is partially supported by CNPq/Brazil grant 303762/2019-5 and FAPESP/Brazil grant
2019/02512-5.
The first author would like to thank Prof. Seungly Oh for many helpful discussions concerning \cite{Flores OH and Smith}.


\end{document}